\documentclass[11pt,a4paper]{article}
\usepackage{bbm}
\usepackage{graphicx}
\graphicspath{ {./images/} }
\usepackage{wrapfig}
\usepackage[mathscr]{euscript}
\usepackage{subcaption}
\usepackage{hyperref}
\usepackage{epstopdf}
\usepackage{enumerate}
\usepackage{amsmath,amsfonts,amssymb,amsthm,epsfig,epstopdf,titling,url,array}
\usepackage{color}
\usepackage{framed}
\usepackage[utf8]{inputenc}
\usepackage[english]{babel}

\usepackage{xparse}

\usepackage{float}
\usepackage{hyperref}
\usepackage{authblk}

\newtheorem{theorem}{Theorem}[section]

\newtheorem{definition}{Definition}[section]
\newtheorem{lemma}[theorem]{Lemma}

\newtheorem{Observation}{Observation}[section]
\newtheorem{Remark}{Remark}[section]

\newtheorem{discu}{Discussion:}
{\tiny }

\numberwithin{equation}{section}
\newtheorem{conje}{Conjecture:}

\usepackage{multirow}
\usepackage{longtable}
\begin{document}
	\title{Connectedness of independence attractors of graphs with independence number three}
	\author[1]{Moumita Manna \footnote{Corresponding author, s23ma09002@iitbbs.ac.in}}
	\author[2]{Tarakanta Nayak \footnote{tnayak@iitbbs.ac.in}}
	\affil[1,2]{Department of Mathematics, Indian Institute of Technology Bhubaneswar, Bhubaneswar, India}
	\date{}
	\maketitle
	\begin{abstract}
		An independent set in a simple graph $G$ is a set of pairwise non-adjacent vertices in $G$. The independence polynomial of $G$, denoted by $I_G$ is defined as $1 + a_1 z + a_2 z^2+\cdots+a_d z^{d}$, where $a_i$ denotes the number of independent sets  with cardinality $i$ and $d$ is the cardinality of a largest independent set in $G$. This $d$ is  known as the independence number of $G$.  Let $G^m$ denote the $m$-times lexicographic product of $G$ with itself. The independence attractor  of $G$, denoted by   $\mathcal{A}(G)$  is defined as  $\mathcal{A}(G) =  \lim\limits_{m\rightarrow \infty} \{z: I_{G^m}(z)=0\}$, where the limit is taken with respect to the Hausdorff metric defined on the space of all compact subsets of the plane. 
		This paper investigates the connectedness  of the independence attractors of all graphs with independence number three. Let  the independence polynomial of $G$ be $1+a_1 z +a_2 z^2 +a_3 z^3$.   For $a_1 =3$, $\mathcal{A}(G)$ turns out to be $ \{-1\}$. For $a_1 >3$, we prove the following.
		If $a_2 ^2 \leq 3 a_1 a_3$, or $3 a_1 a_3 < a_2 ^2 < 4a_3 (a_1 -1)$ then $\mathcal{A}(G)$ is totally disconnected. For $a_2 ^2 =4a_3 (a_1 -1) $,  $\mathcal{A}(G)$ is connected when $a_1 =5$ and is disconnected but not totally disconnected for all other values of $a_1$. If $a_2 ^2 > 4a_3 (a_1 -1)$ then $\mathcal{A}(G)$ can be connected, totally disconnected or disconnected but not totally disconnected depending on further conditions involving $a_1, a_2$ and $a_3$.  Examples of graphs exhibiting all the possibilities are provided.

	\end{abstract}
	\section{Introduction}
	
	Let  $G$ be a simple graph. An independent set in $G$ is a set of pairwise non-adjacent vertices in $G$.  By an $i$-independent set, we mean an independent set with exactly $i$ elements. The independence polynomial of $G$, denoted by $I_G$ is defined as $1+a_1 z +a_2 z^2+ \cdots + a_d z^d$ where $a_i$ denotes the number of $i$-independent sets  and $d$ is the cardinality of a largest independent set.  This $d$ is  known as the independence number of $G$.  Note that $a_1$ is the number of vertices of $G$.
	Independence polynomials appear as the partition function of the hard-core model  in statistical physics (\cite{scot-sokal2005,leeyang1952}). Information on their roots, also called independence roots is recently shown to be crucial in designing efficient algorithms to compute the values of the polynomials (\cite{ivona-2020}). Regions in the plane that do not contain any independence root for a family of graphs are related to phase transition. Such zero-free regions for graphs with a given maximum vertex degree   have been studied in (\cite{hans-2019}). 
	\par  Since all the coefficients of an independence polynomial are positive, it cannot have any positive root. On the other hand, families of graphs are found whose independence roots are dense in $\mathbb{C}$ (\cite{brownetal2004}). This article deals with the limit set of independence roots of a sequence of graphs arising out of lexicographic product.  
	
	The  lexicographic product of a simple graph $G$ having vertex set $V(G)$ with itself is the graph with the vertex set $V(G)\times V(G)$ and such that a vertex $(a,x)$ is adjacent to a vertex $(b,y)$ if and only if $a$ is adjacent to $b$ in $G$, or $a=b$ and $x$ is adjacent to $y$ in $G$.  This amounts to replacing each vertex of $G$ with a copy of $G$ and joining the vertices belonging to two such different copies whenever the underlying vertices are adjacent. For a natural number $m$, let $G^m$ denote the $m$-times lexicographic product of   $G$ with itself. 
	For a polynomial $f$, let  $\text{Roots}(f)=\{z: f(z)=0\}$.
	The set $ \text{Roots} (I_{G^m})$ is finite and therefore a compact subset of the plane for each $m$. The limit  of $ \text{Roots} (I_{G^m})$ as $m \to \infty$ with respect to the Hausdorff metric is known to exist  (see Section \ref{prelim}), and is called the \textit{independence attractor}  $\mathcal{A}(G)$ of $G$(see ~\cite{moumita-2025}).
	\begin{definition}[Independence attractor]
		$\mathcal{A}(G) =\lim\limits_{m \to \infty} \text{Roots} (I_{G^m}). $ 
	\end{definition}
	For $\epsilon>0$ and a compact subset $K$ of the plane, let $[K]_\epsilon$ be the union of all the disks with center at some point of $K$ and radius $\epsilon$.  Then the Hausdorff distance between two compact subsets $K_1,K_2$ of the plane is the minimum of the set $\{\epsilon: K_2 \subseteq [K_1]_\epsilon ~\mbox{and}~K_1 \subseteq [K_2]_\epsilon\}$.
	Therefore, the independence attractor of a graph gives the location of the independence roots of $G^m$  up to arbitrary closeness for all sufficiently large $m$. Consequently, zero-free regions for the graphs in  this sequence (except finitely many) are obtained. For an empty graph (i.e., without any edge) $G$ on $n$ vertices,  $I_G (z)= (1+z)^n$ and therefore $\mathcal{A}(G) =\{-1\}$.  The independence polynomial of a complete graph is linear  and Hickman observed that the independence attractor is $\{0\}$ (see \cite{hickman2001}).  In the same paper, the author  describes the connectedness of independence attractors of graphs with independence number two. Graphs  whose  independence attractors are topologically simple (i.e., lines and circles) are characterized (see \cite{line}, \cite{moumita-2025}). We take up graphs with independence number three and determine the connectedness of their independence attractors.
	\par 
	
	In order to state the results obtained in this article, let $G$ be a graph with independence number three and its independence polynomial be $$ I_G (z)=1+a_1 z+a_2 z^2 +a_3 z^3.$$ As there is at least one $3-$independent set (i.e., an independent set with three elements), we have $a_1 \geq 3$. There are three cases depending on the nature of critical points of $I_G$, i.e., zeros of $I_G '$, which are given by  
	$$c_1 =\frac{-a_2 - \sqrt{a_2 ^2 -3 a_1 a_3}}{3a_3}~\mbox{and} ~ c_2 =\frac{-a_2 + \sqrt{a_2 ^2 -3 a_1 a_3}}{3a_3}.$$   
	\begin{enumerate}
		\item \textbf{Bicritically non-real} ($a_2 ^2 < 3a_1 a_3$): There are two distinct non-real critical points and both have the same real part.
		\item   \textbf{Unicritical} ($a_2 ^2=3a_1 a_3$): There is a single critical point, namely $\frac{-a_2}{3a_3}$.
		
		\item \textbf{Bicritically real} ($a_2 ^2> 3a_1 a_3$): There are two distinct real critical points and both are negative.
		
	\end{enumerate} 
	For the sake of brevity, we refer to a graph with independence number three as \textit{bicritically non-real}, \textit{unicritical} or \textit{bicritically real} if the coefficients of its independence polynomial $1+a_1z +a_2z^2+ a_3z^3$ satisfy $a_2 ^2 < 3a_1 a_3,  a_2 ^2 = 3a_1 a_3$ or $a_2 ^2 > 3a_1 a_3$ respectively. A subset $K$ of the plane is called totally disconnected if each maximally connected subset of $K$ is a singleton.  
	\begin{theorem}[Bicritically non-real]
		The independence attractor of every  bicritically non-real graph is totally disconnected. 
		\label{bicritically non-real}
	\end{theorem}
	\begin{theorem}[Unicritical] The independence attractor  of a unicritical graph is $\{-1\}$ if the graph has three vertices, and is  totally disconnected otherwise. 
		\label{unicritical}
	\end{theorem}
	There is only one graph with independence number three that has three vertices. This is unicritical (see Lemma~\ref{basic-obs}). Further, it is observed from the proof of Theorem~\ref{unicritical} that a unicritical graph cannot have four vertices.
	\par 
	A point $z_0$ is called a fixed point of a polynomial $f$ if $f(z_0)=z_0$. 
	For bicritically real graphs, there are three cases depending on the nature of the fixed points of  $I_G(z)-1$. The point $0$ is always a fixed point. The non-zero  fixed points of $I_G (z)-1$ are given by $$\frac{-a_2\pm \sqrt{a_2 ^2 -4a_3 (a_1 -1)}}{2 a_3}.$$
	
	It can be seen that there are two non-real fixed points, a single fixed point  or two real fixed points of $I_G(z)-1$ when $a_2 ^2 < 4 a_3 (a_1 -1)$, $a_2 ^2 = 4 a_3 (a_1 -1)$ or  $a_2 ^2 > 4 a_3 (a_1 -1)$ respectively. The next theorem deals with the first two cases.	
	\begin{theorem}[Bicritically real - two non-real fixed points or one fixed point] Let  $G$ be a bicritically real  graph and $I_G(z)=1+a_1 z+a_2 z^2 +a_3 z^3$.  
		\begin{enumerate}
			\item If $a_2 ^2 <4a_3 (a_1 -1) $ then  $\mathcal{A}(G)$ is totally disconnected.
			\item If $a_2 ^2 =4a_3 (a_1 -1) $ then $\mathcal{A}(G)$ is connected for $a_1 =5$, and disconnected but not totally disconnected otherwise.
		\end{enumerate}  
		\label{bicritically real-two non-real fixed points or one fixed point}
	\end{theorem}
	For $a_2 ^2 > 4a_3 (a_1 -1)$, $I_G (z)-1$ has two non-real zeros, a single zero or  two real zeros   if $ a_2 ^2 <4a_1 a_3$, $ a_2 ^2 =4a_1 a_3$ or $ a_2 ^2 >4a_1 a_3$ respectively. Of course, $0$ is always a zero of $I_G (z)-1$. 
	\begin{theorem}[Bicritically real - two real fixed points]
		Let  $G$ be a bicritically real  graph and $I_G(z)=1+a_1 z+a_2 z^2 +a_3 z^3$  such that $a_2 ^2 > 4a_3 (a_1 -1)$.
		\begin{enumerate}
			
			\item  Let $ a_2 ^2  < 4a_1 a_3$.
			
			If $a_1=4$, then $\mathcal A(G)$ is disconnected but not totally disconnected. If $a_1>4$, then we have the following.  
			\begin{enumerate}
				\item If $4a_3 (a_1 -1) < a_2 ^2  < \frac{a_3 (2a_1 -3)^2}{a_1 -2}$ then $\mathcal{A}(G)$  is disconnected but not totally disconnected.
				\item If $a_2^2=\frac{a_3(2a_1-3)^2}{a_1-2}$  then $\mathcal{A}(G)$ is disconnected but not totally disconnected.
				\item If $\frac{a_3(2a_1-3)^2}{a_1-2}<a_2^2<\frac{4a_3(a_1-2)^2}{a_1-3}$ then   $\mathcal{A}(G)$ is disconnected but not totally disconnected.
				\item If $a_2^2=\frac{4a_3(a_1-2)^2}{a_1-3}$  then for $a_1 \leq 7$,  $\mathcal{A}(G)$ is connected, and for $a_1 >7$, it is disconnected but not totally disconnected.
				\item  If $\frac{4a_3(a_1-2)^2}{a_1-3}<a_2^2<4 a_1 a_3$ then  $\mathcal{A}(G)$ is disconnected except for two possibilities, namely  $(a_1, a_2, a_3)=(7,9,3)$ or $(8, 11,4)$.
				
			\end{enumerate} 
			\item If $ a_2 ^2 =4a_1 a_3$ then $\mathcal{A}(G)$  is connected for $a_1 \leq 9$ and totally disconnected otherwise.
			\item  If $ a_2 ^2  > 4a_1 a_3$ then $\mathcal{A}(G)$ is disconnected.
		\end{enumerate} 
		\label{bicritically real-two real fixed points}
	\end{theorem}
	In Theorems~\ref{bicritically real-two real fixed points}(1)(e) and \ref{bicritically real-two real fixed points}(3), the disconnected $\mathcal{A}(G)$ can actually be totally disconnected. Such examples are given in  Remarks~\ref{totallydisconnected-example-9-13-5} and ~\ref{totallydisconnected-example-7-6-1}.

	The Julia 	set of a polynomial $f$ with degree at least two, denoted by $\mathcal{J}(f)$ is defined as the boundary of the filled-in Julia set $K(f) =\{z: \{f ^n (z)\}_{n>0}~\mbox{ is bounded}\} $. The reduced independence polynomial of a graph $G$ is defined as  $I_G(z)-1$. It is denoted by $P_G (z)$.  The relation between   $\mathcal{J}(P_G)$ and  $\mathcal{A}(G)$ is crucial and is used in the proofs of the above mentioned theorems. The following result makes this relation precise for all graphs with independence number three.
	\begin{theorem} For a non-empty graph $G$ with independence polynomial $I_G (z)=1+a_1 z+a_2 z^2 +a_3 z^3$, 
		$\mathcal{A}(G)$ is the disjoint union of $  \mathcal{J} (P_G) $ and~ $\bigcup_{k\geq 1}~\mbox{Roots}~(I_{G^{k}}) $	if and only if $a_1 \geq 3, a_2 =2a_1 -3$ and $a_3 = a_1 -2$. In all other cases,  $\mathcal{A}(G) = \mathcal{J}(P_G)$.
		\label{IA-JS}
	\end{theorem} 
	For different values of $a_1$ in Theorem~\ref{IA-JS}, we have some remarks.
	\begin{Remark}
		\begin{enumerate}
			\item If $a_1 =3$ then $I_G (z)=1+ 3 z+ 3 z^2 + z^3$ and  clearly, $a_2 =2a_1 -3$ and $a_3 = a_1 -2$. Theorem~\ref{unicritical} deals with this situation. For $a_1 >3$, $a_2 =2a_1 -3$ and $ a_3 =a_1 -2$ imply that $a_2^2=\frac{a_3(2a_1-3)^2}{a_1-2}$. This is described in Theorem~\ref{bicritically real-two real fixed points}(1)(b).
			However, the condition $a_2^2=\frac{a_3(2a_1-3)^2}{a_1-2}$ does not always ensure that $a_2 =2a_1 -3, a_3 =a_1 -2 $, e.g., $(a_1, a_2, a_3) =(12, 42, 40)$.  
			\item  Let $a_2^2=\frac{a_3(2a_1-3)^2}{a_1-2}$. If $a_1 \leq 8$ then $-1$ is a multiple root of $I_G$ (see the proof of Theorem~\ref{bicritically real-two real fixed points}(1(b))) and consequently,   $\mathcal{A}(G)$ is the disjoint union of $\mathcal{J} (P_G)$ and $ \bigcup_{k\geq 1}~\mbox{Roots}~(I_{G^{k}})$. If $a_1 >8$ then   both the possibilities, as stated in Theorem~\ref{IA-JS} can occur. For example, $(a_1, a_2, a_3)=(12, 42,40)$ gives that $\mathcal{A}(G) =\mathcal{J} (P_G)$ whereas for $(a_1, a_2, a_3) =(9, 15,7)$,  the other possibility holds.
		\end{enumerate} 
		\label{IA-JS-rem} 
	\end{Remark}
	Examples of graphs satisfying the hypotheses of Theorems~\ref{bicritically non-real}, \ref{unicritical}, \ref{bicritically real-two non-real fixed points or one fixed point} and \ref{bicritically real-two real fixed points} are provided. These examples also show  that there is a sequence of graphs in each case for which the independence attractor has arbitrarily large diameter. The diameter of a set $A \subset \mathbb{C}$ is the supremum of the set $\{|z-w|: z, w \in A\}$.
	\begin{theorem}
		There exist three sequences $\{G_1 (n)\}_{n>0}, \{G_2 (n)\}_{n>0}$ and $\{G_3 (n)\}_{n>0}$  of bicritically non-real, unicritical and bicritically real graphs respectively such that $diam (\mathcal{A}(G_i (n))) \to \infty$ as $n \to \infty$ for each $i=1,2,3$. 
		\label{examples}
	\end{theorem}
	
	Several other examples can also be constructed. 
	\par  
	Some known as well as new results on independence attractors and polynomial Julia sets including the proof of  Theorem~\ref{IA-JS} are provided in Section~\ref{prelim}. Section 3 contains the proofs of Theorems~\ref{bicritically non-real}, \ref{unicritical}, \ref{bicritically real-two non-real fixed points or one fixed point}, ~\ref{bicritically real-two real fixed points} and ~\ref{examples}.
	By saying a reduced independence polynomial, we mean a reduced independence polynomial of some graph throughout this article. 
	
	\section{Independence attractors and polynomial Julia sets}\label{prelim}
	
	A key fact relating to the lexicographic product of a graph with itself and its independence polynomial is given by  $I_{G^2}(z) =I_G (I_G (z)-1)$ (see Theorem 1.1,~\cite{brownetal2003}). Observe that  $I_{G^2}(z) =P_G (P_G (z))+1=P_G ^2 (z)+1$, and in general, $I_{G^m}(z) =P_G ^m (z)+1$. In other words,
	\begin{equation}
		Roots(I_{G^m})  =\{z: P_G^m (z)=-1\}~\mbox{for each}~ m.
	\end{equation}   
	By virtue of this equality, the iterative behavior of $P_G$ assumes significance in the study of the independence attractor of $G $. This simple but very important fact  gives rise to the following equivalent definition of the independence attractor.
	\begin{definition}$\mathcal{A}(G) =  \lim\limits_{m \to \infty}   \{z: P_G^m (z)=-1\}$.
	\end{definition}
	
	It is   known that the Julia set $\mathcal{J}(P_G)$ of $P_G$ is the same as $\lim\limits_{k \to \infty} {Roots}~(P_{G^{k}})$  (see Theorem 3.3, \cite{brownetal2003}). This is probably the reason why $\mathcal{J}(P_G)$ is referred to as the independence fractal of $G$.  Hickman  proved that  $\mathcal{A}(G) \supseteq \mathcal{J}(P_G)$ for every graph $G$, and the equality holds almost always (see\cite{hickman2001},\cite{moumita-2025}). To state this result, recall that ${Roots}~(I_{G^{k}}) =\{z:I_{G^{k}}(z)=0\}$.  
	%
	%
	%
	\begin{lemma}\label{JS-attractor} 
		Let $G$ be a non-empty graph. 
		\begin{enumerate}
			\item If  $I_G (-1) \neq 0$, or  $I_G (-1) = 0$ with $I_G ' (-1) \neq 0$  then $\mathcal{A}(G)=\mathcal{J}(P_G) $.
			\item If  $I_G (-1) = 0$ with $I_G ' (-1) = 0$  then $\mathcal{A}(G)$ is the disjoint union of $  \mathcal{J} (P_G) $ and $ \bigcup_{k\geq 1}~\mbox{Roots}~(I_{G^{k}})$.
		\end{enumerate}
		
	\end{lemma}
	Lemma~\ref{JS-attractor} leads to a proof of Theorem~\ref{IA-JS}.
	\begin{proof}[Proof of Theorem~\ref{IA-JS}]
		Let $I_G(z)=1+a_1 z+a_2 z^2 +a_3 z^3$. If $I_G (-1)=0$ and $I_G '(-1)=0$ then   $a_3 = a_1 -2$ and $ a_2 =2 a_1 -3$. Converse is clearly true, and we are done by Lemma~\ref{JS-attractor}. 
	\end{proof}
	In view of Lemma~\ref{JS-attractor}, the study of $\mathcal{A}(G)$ calls for understanding the Julia set of $P_G$. Recall that the Julia set of a polynomial $f$ with degree at least two, denoted by $\mathcal{J}(f)$ is defined as the boundary of the filled-in Julia set $K(f) =\{z: \{f ^n (z)\}_{n>0}~\mbox{ is bounded}\} $. The Fatou set of $f$ is defined as $\widehat{\mathbb{C}} \setminus \mathcal{J}(f)$. Equivalently, the Fatou set of $f$ is the set of all points in $\mathbb{C} \cup \{\infty\}$ where the sequence $\{f^n\}_{n>0}$ is equicontinuous.
	The set $\{z: f^n (z) \to \infty~\mbox{as}~ n \to \infty\}$ is open and connected and is  known as the basin of $\infty$. This is always contained in the Fatou set of $f$. The Julia set is also the same as the boundary of the basin of $\infty$.  
	
	\par Recall that a point $z_0$ for which $f(z_0)=z_0$ is called a fixed point of $f$. It is called attracting, neutral or repelling if $|f'(z_0)|<1, =1$ or $>1$ respectively. If $f'(z_0)=0$ then $z_0$ is called a super-attracting fixed point of $f$.  Repelling and parabolic fixed points (i.e., neutral fixed points for which $f'(z_0)$ is a root of unity) are always in the Julia set whereas attracting fixed points are known to be in the Fatou set of the polynomial.
	For an attracting or a parabolic fixed point $z_0$ of $f$, the set  $\{z: f^n(z) \to z_0 ~\mbox{as}~n \to \infty\}$ is known as the attracting  or parabolic basin of $z_0$ respectively. The connected component of the attracting basin of $z_0$ containing $z_0$ is called the attracting domain of $z_0$. Similarly, the connected subset of the parabolic basin of $z_0$ whose boundary contains $z_0$ is called a parabolic domain of $z_0$. There can be more than one parabolic domain corresponding to a parabolic fixed point.  All these facts can be found in \cite{Beardon_book}.
	
	\par 
	We now present a result  on the topology of the Julia set of a polynomial by putting Theorems 9.5.1 and 9.8.1 of \cite{Beardon_book} together. The forward orbit of a point $z$ under a polynomial $f$ is the set $\{f^n (z): n >0\}$.
	\begin{lemma}
		Let $f$ be a polynomial with degree at least two. 
		\begin{enumerate}
			\item Its Julia set is connected if and only if the forward orbit of each of its critical points is bounded.
			
			\item Its Julia set is totally disconnected if the forward orbit of each of its critical points is unbounded. 
		\end{enumerate} 
		\label{connected-Juliaset}
	\end{lemma}
	Some basic properties of $\mathcal{J}(P_G)$ are proved in the following lemma, which is slightly more general than what is required. It is pertinent to mention that the forward orbit of a point under a polynomial necessarily tends to $\infty$ whenever it is unbounded. We say a point escapes (under the iteration of a polynomial) if its forward orbit is unbounded. 
	
	\begin{lemma}\label{Juliaset-prop}
		Let $f$ be a polynomial with positive coefficients and degree at least two.  Then, $\lim \limits_{n \to \infty}f^n (z) = \infty$ if and only if $\lim \limits_{n \to \infty} f^n(\overline{z}) = \infty$ and consequently, $z \in \mathcal{J}(f)$ if and only if  $\overline{z} \in \mathcal{J}(f)$. Further, if    $f(0)=0$ and $f'(0)>1$ then the following are true.
		\begin{enumerate}
			\item The Julia set of $f$ contains $0$.
			\item  $\lim\limits_{n \to \infty} f^n (x) = +\infty$ for each $x>0$ and consequently, the Julia set of $f$ does not intersect the positive real axis.
			\item If $f$ has no negative fixed point  and all its critical points are real then  the Julia set of $f$ is totally disconnected.
		\end{enumerate} 
	\end{lemma}	
	\begin{proof}
		Since all the coefficients of $f$ are real, $f(\overline{z})=\overline{f(z)}$. Consequently, $f^n(\overline{z})=\overline{f^n(z)}$ for all $n$.
		Therefore  $\lim \limits_{n \to \infty}f^n (z) = \infty$ if and only if $\lim \limits_{n \to \infty} f^n(\overline{z}) = \infty$. It also follows that $\{f^n\}_{n>0}$ is equicontinuous at $z$ if and only if it is equicontinuous at $\overline{z}$.
		In other words, $z\in \mathcal{J}(f)$ if and only if  $\overline{z}\in \mathcal{J}(f) $. 
		
		\begin{enumerate}
			\item  Since  $f(0)=0$  and $f'(0)>1$, $0$ is a repelling fixed point of $f$ and therefore it is in the Julia set. 
			\item  Consider $g(x)= f(x)-x$. Then $g'(x)=f'(x) -1$ and $g''(x)=f''(x) >0$ for each $x>0$ since all the coefficients of $f$  are positive. This  gives that $g'$ is strictly increasing and since $g'(0)=f'(0)-1>0$, $g'(x)>0$ for each $x>0$. In other words, $g$ is strictly increasing in $(0, \infty)$. As $g(0)=0$, we have $f(x)>x$ for all $x>0$. In particular, there is no fixed point of $f$ in the positive real axis. Further, it follows that $\{f^n(x)\}_{n>0} $ is a strictly increasing sequence for $x>0$. This can not be bounded above, since that would imply its convergence and the limit point has to be a positive fixed point of $f$.  Therefore  $\lim\limits_{n \to \infty} f^n (x) = +\infty$ for each $x>0$. This means that $(0,\infty)$ is in the basin of $\infty$ and hence is contained in the Fatou set of $f$.
			
			\item 	If   $f(x_1) < x_1$ and $f(x_2) > x_2$ for two distinct negative real numbers $x_1, x_2$  then  by the Intermediate Value Theorem, there is a point $x^* <0 $ such that $f(x^*) =x^*$.  However, this is not true by our assumption. Therefore,  $f(x)>x$ for all $x<0$ or $f(x)<x$ for all $x<0$.
			
			\par 
			Since $g'(0)>0$ and $g'$ is continuous, there is a $\delta>0$ such that $g'(x)>0$ for all $x \in (-\delta, \delta)$. This means that $g$ is strictly increasing in $(-\delta, \delta)$. As $g(0)=0$, $f(x)<x$  for all $x \in (-\delta, 0)$. It follows from the conclusion of the previous paragraph that $f(x)<x$ for all $x<0$.
			Now $\{f^n(x)\}_{n>0} $ is a strictly decreasing sequence and therefore $\lim\limits_{n \to \infty} f^n (x) = -\infty$ as there is no negative fixed point for $f$.
			\par It is shown that all non-zero real numbers escape under the iteration of $f$. Since all the critical points of $f$ are non-zero and  real, their forward orbits are unbounded. The Julia set of $f$ is totally disconnected by Lemma~\ref{connected-Juliaset}(2).
		\end{enumerate}
		
	\end{proof}
	That the  derivative of a polynomial on its filled-in Julia set is related to the connectedness of the Julia set follows from  a nice result of Buff \cite{Buff-2002}.
	\begin{lemma}	If $f_d$ is a polynomial of degree $d \geq 2$ with connected filled-in Julia set  $K(f_d)$ then  $|f_{d}'(z)| \leq  d^2$ for all  $z \in K(f_d)$. Equality holds if and only if the Julia set of $f_d$ is a line segment.\label{Buff-2003}
	\end{lemma}
	If $0$ is a  fixed point of $f_d$ then it is in $K(f_d)$. The Julia set of $f_d$ is disconnected whenever $|f_d '(0)| > d^2$. These are going to be used for studying the Julia set of $P_G (z)=a_1 z+a_2 z^2 +a_3 z^3$, and it is seen that there are only finitely many situations where $\mathcal{J}(P_G )$ is connected.
	
	We need the following from ~\cite{Pom 1975} in order to get the location and size of connected Julia sets of reduced independence polynomials. A polynomial is called monic or centered if its leading coefficient is $1$ or its second leading coefficient  is $0$ respectively.
	\begin{lemma}
		If  $f$ is a monic and centered polynomial with degree at least two and $K(f)$ is connected then  $K(f) \subset \{z: |z| \leq 2 \} $.
		\label{size-JS}
	\end{lemma}
	Applying Lemma ~\ref{size-JS} to a reduced independence polynomial, we get the following. We use $P$ instead of $P_G$ here and elsewhere for the sake of notational simplicity.
	\begin{lemma}\label{filled-Julia-estimate}
		Let  $P(z)=a_1 z+ a_2 z^2 +\cdots + a_{d-1} z^{d-1}+a_d z^d , d \geq 2$ be a reduced independence polynomial with connected $K(P)$. Then   $K(P) \subset \{z: |z +\frac{a_{d-1}}{d a_d}| \leq 2 \left(\frac{1}{a_d}\right)^{\frac{1}{d-1}}\}$. In particular, if $P$ is cubic then $K(P) \subset \{z: |z +\frac{a_{2}}{3 a_3}| \leq  \frac{2}{\sqrt{a_3}}\}$.
	\end{lemma}

\begin{proof}
	Let $\phi(z)= az -\frac{a_{d-1}}{d a_d}$ where $a$ is a $(d-1)$-th root of $\frac{1}{a_d}$. Then it can be seen that (for example see Lemma 2.2, \cite{rotational-symmetry-2024}) $\phi^{-1} \circ P\circ \phi$ is a monic and centered polynomial. Since $K(P)= \phi (K(\phi^{-1} \circ P\circ \phi ))$ and $K(\phi^{-1} \circ P\circ \phi ) \subset \{z: |z| \leq 2\}$, we have that $K(P)$ is contained in the image of $\{z:|z| \leq 2\}$ under $\phi$, which is nothing but  $ \{z: |z +\frac{a_{d-1}}{d a_d}|  \leq 2 \left(\frac{1}{a_d}\right)^{\frac{1}{d-1}}\}$. 
	\par If $P$ is cubic then $d=3$ and $K(P)\subset \{z\in\mathbb{C}:|z+\frac{a_2}{3a_3}|\leq\frac{2}{\sqrt{a_3}}\}$. \end{proof}
For using later, we make a definition.
\begin{definition}(Critical disk)
	The critical disk for  a cubic polynomial $P(z)= a_1 z +a_2 z^2 +a_3 z^3$, denoted by $D_P$ is defined as $\{z: |z+\frac{a_2}{3a_3}| \leq  \frac{2}{\sqrt{a_3}}\}$. 
	\label{criticaldisk}
\end{definition}
Every point outside the critical disk of a polynomial with the connected filled-in Julia set belongs to the basin of $\infty$.

Now, we look for possible cubic reduced independence polynomials.
\begin{lemma} If  $ P(z)= a_1 z+a_2 z^2+  a_3 z^3$ is a cubic reduced independence polynomial, then $ 3 \leq a_2 \leq \frac{a_1 (a_1 -1)}{2}$ and $1 \leq a_3 \leq \frac{a_1 (a_1 -1)(a_1 -2)}{6}$.
	Further, if $a_1 =3$ then $(a_3,a_2) =(1, 3)$, and $P$ has only one critical point. 
	\label{basic-obs}
\end{lemma}
\begin{proof}
	As  $a_3 \geq 1$, we have $a_2 \geq 3$. By definition, $a_2 \leq \frac{a_1 (a_1 -1)}{2} $ and $1 \leq a_3 \leq \frac{a_1 (a_1 -1)(a_1 -2)}{6}$. If $a_1 =3$ then $(a_3,a_2) =(1, 3)$ and $-1$ is the only critical point of $P$. 
\end{proof}
Recall that for a real number $x$, $\lfloor x\rfloor$ and $\lceil x\rceil$ denote the largest integer at most $x$
and the smallest integer at least $x$, respectively.
\begin{lemma}
	Let $P(z)=a_1z+a_2z^2+a_3z^3$ be the reduced independence polynomial. Then the following statements hold.
	\begin{enumerate}
		\item  For $a_1=4,$ we have $a_2<6$. Further,   $(a_2, a_3) \in \{(3,1), (4,1), (5,2)\}$.
		\item  For $a_1=5,$ we have $a_2<9.$  Further, if $a_2=8$, then $a_3=4$.
		\item  For $a_1=6$, we have $a_2<13$.  Further, if  $a_2=12$, then $a_3=8$.
		\item  For $a_1=7,$ we have $a_2<17.$  Further, if  $a_2=16$ then $a_3=12$.
		\item  For $a_1=8,$ we have $a_2<22.$  Further, if  $a_2=21$ or $20$  then   $a_3=18$ or $a_3 \in \{15,16\}$ respectively. 
	\end{enumerate}
	\label{basic-independentpoly}
\end{lemma}
\begin{proof}
	Let $G$ be a graph with $P_G=P$. Then $G$ has $a_1$ vertices, independence number three, $a_2$ pairs of non-adjacent vertices and $a_3$ number of $3$-independent sets. Further, the number of edges $e=\binom{a_1}{2}-a_2$.
	Tur\'an's theorem states that every graph with $n$ vertices not containing a complete graph on $r+1$ ($r+1\le n$) vertices as a subgraph has at most as many edges as the Tur\'an graph $T(n,r)$ (see Chapter VI in \cite{Bollobas}). The complement of $G$ does not contain any $K_4$ and has $a_2$ edges. Hence for $a_1\ge 4$, we obtain \begin{equation}
a_2\le \frac{a_1^2}{3}.\label{Eq2.2}
	\end{equation} 
	
	Further, we require the following observations on graphs with the above mentioned properties to proceed further.
	\begin{enumerate}
		\item[(O1)] The independence number of $G$ is the sum of the independence numbers of its connected components. Hence $G$ has at most three components. If $G$ has $k$ components with $i$-th component having $n_i$ vertices, then the $i$-th component has at least $n_i-1$ edges and therefore $e\ge\sum_{i=1}^{k}(n_i-1)=a_1-k$. In other words, $k\ge a_1-e$.
		\item[(O2)] If $G$ has three components, then each is complete with number of vertices, say $n_1,n_2,$ and $n_3$. Consequently, $e=\sum_i\binom{n_i}{2}$ where $\binom{1}{2}$ is taken as $0$. Every $3$-independent set has exactly one vertex in each component, and thus $a_3=n_1n_2n_3$.
		\item[(O3)] If $G$ has two components, then one is a complete graph $K_s$ and the other is a connected graph $G'$ on $b=a_1-s$ vertices with independence number $2$. In this situation $b$ must be at least $3$. The complement of $G'$ is triangle-free, so it has at most $\lfloor b^2/4\rfloor$ edges by Tur\'an's theorem, with equality only for the complete bipartite graph $K_{\lfloor b/2\rfloor,\lceil b/2\rceil}$ (\cite{Bollobas}). The complement of $G'$ cannot be this graph, since then $G'$ would be disconnected. Therefore
		$
		|E(G')|\geq \binom{b}{2}-\Big\lfloor \tfrac{b^2}{4}\Big\rfloor+1=\binom{a_1-s}{2}-\Big\lfloor \tfrac{(a_1-s)^2}{4}\Big\rfloor+1.$ For using later (4 and 5 of this lemma), let $g(s)=\binom{a_1-s}{2}-\Big\lfloor \tfrac{(a_1-s)^2}{4}\Big\rfloor+1$. 
		Every $3$-independent set in $G$ consists of a vertex of $K_s$ and two non-adjacent vertices of $G'$. Hence $a_3=s\big(\binom b2-|E(G')|\big)$.
		\item[(O4)] If $G$ is connected and has $m$ vertices and at most $m$ edges, then
		$G$ has an independent set of size at least $\lceil (m-1)/2\rceil$.
		
		Indeed, since $G$ is connected, it contains a spanning tree $T$, which has
		$m-1$ edges. As $G$ has at most $m$ edges, either $G=T$ or $G$ id obtained from $T$ by joining two of its non-adjacent vertices by an edge. In the second case $G$ contains exactly one cycle.
		
		Let $G=T$ then it is bipartite. In the other case,  let $w$ be a vertex of the unique cycle of $G$. Then $ G-w$ does not contain any cycle, and hence it is also bipartite. In both the cases, $G$ has an induced bipartite sub-graph  with $m$ or $m-1$ vertices. Let $X, Y$ denote the corresponding partitions of vertices of the  bipartite sub-graph. Then both $X, Y$ are independent sets in $G$. As
		$|X|+|Y| \ge m-1$, one of them has at least $\lceil (m-1)/2\rceil$ many vertices.

	\end{enumerate}
	
	\begin{enumerate}
	\item For $a_1=4$, we have $a_2<6$ by Equation~\ref{Eq2.2}. If $a_2=3$, then the graph has $3$ edges, giving $a_3=1$. If $a_2=4$, then the graph has $2$ edges and these edges have a common vertex of incidence, i.e., $a_3=1$. If $a_2=5$, then the graph has only $1$ edge, giving $a_3=2$.

	\item For $a_1=5$, we have $a_2<9$ by Equation~\ref{Eq2.2}. If $a_2=8$, then $e=2$. So $G$ has at least $a_1-e=3$ components by (O1), and therefore it has exactly three components. By (O2), each component of $G$ is complete and the number of vertices in these three components  is either $(3,1,1)$ or $(2,2,1)$. However, the possibility $(3,1,1)$ leads to three edges, which is not true. Therefore, $a_3=4$.
	
	\item For $a_1=6$, we have $a_2<13$ by Equation~\ref{Eq2.2}. If $a_2=12$, then $e=3$. By (O1), $G$ has at least $3$ components, and therefore it has exactly $3$ components. By (O2), each component of $G$ is complete. Further, the number of vertices in these three components is $(4,1,1)$, $(3,2,1)$ or $(2,2,2)$ leading to the total number of edges in $G$ as $6,4$ or $3$ respectively. Since $e=3$, each component contains exactly two vertices and therefore $a_3=8$.

	\item For $a_1=7$, we have $a_2<17$ by Equation~\ref{Eq2.2}. 
	 If $a_2=16$, then $e=5$. By (O1), $G$ has at least two components.
\par 	If $G$ has three components then each is a complete graph. Further, the number of vertices in these three components is $(5,1,1),(4,2,1),(3,3,1)$ or $(3,2,2)$ leading to the total number of edges of $G$ as   $ 10,7,6$ and $5$ respectively. Since $e=5$, one component will have three vertices whereas each of the other two has two vertices. Therefore,  $a_3=12$.
	\par 
If $G$ has two components then one component is $K_s$ with $s\le 3$ by (O3)  (as for $K_4$, we need $6$ edges). Denoting the other component of $G$ by $G'$, we observe that $|E(G')|=5,4$ or $2$
	for $s=1,2$ or $3$ respectively. However, $g(1)=7, \ g(2)=5$ and $g(3)=3$ give that $|E(G')|<g(s)$ for $s=1,2$ and $3$, where $g$ is as defined in (O3). But, this is not possible by (O3).

	\item For $a_1=8$, we have $a_2<22$ by Equation~\ref{Eq2.2}.
	\par
	Let $a_2=21$. Then $e=7$.
	\par 	If $G$ has three components, then  each is a complete graph. Further, the number of vertices in these three components is    $(6,1,1),(5,2,1),(4,3,1),(4,2,2)$ or $(3,3,2)$, leading to the total number of edges of $G$ as   $ 15,11,9,8$ or $7$ respectively. Since $e=7$, one component will have two vertices and the each of the other two will have three vertices. Therefore,  $a_3=18$.
	\par If $G$ has two components then one component is $K_s$ with   $s\le 4$ by (O3). Denoting the other component of $G$ by $G'$, we observe that  $|E(G')|=7,6,4,1 $ for $s=1,2,3,4$, respectively. Taking $g$ as defined in (O3), note that  $g(1)=10,\ g(2)=7,\ g(3)=5,\ g(4)=3$. Thus $|E(G')|<g(s)$ for $s=1,2,3$ and $4$. But this is not possible by (O3).
	\par 
	If $G$ has only one component then it is a tree on $8$ vertices, so $G$ has a $4$-independent set by (O4), a contradiction. 
	\par 
	
	Let $a_2=20$. Then $e=8$. We claim that $a_3\in\{15,16\}$. If $G$ has three components, then each is complete and arguing as above, it can be seen that one component has four vertices and each of the other two has two vertices. Therefore  $a_3=16$ by (O2). Let $G$ have two components. Since $e=8$, using (O3) and arguing as in the similar situation above, we get $K_s$ is a component of $G$ with $s\le 4$. Denoting the other component of $G$ by $G'$, we observe that  $|E(G')|=8,7,5,2$ for $s=1,2,3,4$, respectively. However $g(1)=10,\ g(2)=7,\ g(3)=5,\ g(4)=3$. Thus $|E(G')|<g(s)$ for $s=1$ and $4$, which is not possible by (O3).   Thus the possible values of $s$ are $2$ and $3$, and therefore the corresponding values of $a_3$ are $16$ and $15$ respectively.    
	If $G$ has one component then $G$ is connected with $8$ vertices and $8$ edges. By (O4), $G$ has a $4$-independent set, which is not possible.
\end{enumerate}
\end{proof}

\section{Proofs} 
\subsection{Proof  of Theorem~\ref{bicritically non-real}  }
We require three lemmas to prove Theorem~\ref{bicritically non-real}.
\begin{lemma}
	Let $P$ be a cubic reduced independence polynomial with two non-real critical points. 
	\begin{enumerate}
		\item Then the Julia set of $P$ is either connected or totally disconnected.
		\item If $D_P$ denotes the critical disk of $P$ (see Definition~\ref{criticaldisk}) and  the forward orbit of a critical point intersects the complement of $D_P$ then the Julia set of $P$ is totally disconnected.
	\end{enumerate}\label{totally-disconnected}
\end{lemma}

\begin{proof}
	\begin{enumerate}
		\item 
		For a critical point $c$ of $P$, the forward orbit $\{P^n(c): n \geq 1\}$ of $c$ is bounded if and only if $\{P^n(\overline{c}): n \geq 1\}$ is bounded (by Lemma~\ref{Juliaset-prop}). Note that  $\overline{c}$ is the other critical point of $P$ as all the coefficients of $P'$ are real. If both the critical points remain bounded under the iteration of $P$ then by Lemma~\ref{connected-Juliaset}(1), the Julia set of $P$ is connected. Otherwise both the forward orbits of the critical points are unbounded and therefore the Julia set is totally disconnected by Lemma~\ref{connected-Juliaset}(2). 
		\item 
		
		If the Julia set of $P$ is connected then $K(P) \subset D_P$, by Lemma~\ref{filled-Julia-estimate}. By the hypothesis of this lemma,  the forward orbit of a critical point is in the complement of $D_P$ and hence it is unbounded. It follows from Lemma~\ref{connected-Juliaset}(1) that the Julia set of $P$ is disconnected. 
		The Julia set is totally disconnected by (1) of this lemma.
	\end{enumerate}
\end{proof}

\begin{lemma}
	Let  $P(z)=a_1 z +a_2 z^2 +a_3 z^3$  be a reduced independence polynomial and  $a_2 ^2 < 3 a_1 a_3$. Then the following are true.
	\begin{enumerate}
		\item If  $c = \frac{-a_2 + i \sqrt{3a_1a_3-a_2^2}}{3a_3}$ is the critical point of $P$ with positive imaginary part  then $|P(c)+\frac{a_2}{3a_3}| > \frac{2}{\sqrt{a_3}}$ if and only if $$\left|2 a_2 ^2 - \frac{3 a_3 (a_1 ^2 +6 a_1 -3)}{4} \right| > a_3 \sqrt{\left(\frac{3 (a_1 ^2 +6 a_1 -3)}{4} \right)^2 -12 a_1 ^3 +324}.$$
		\item  $a_1 ^3 -6 a_1 ^2 +9 a_1 -108 > 0$ if and only if $$6  a_1 < \frac{3}{4} (a_1 ^2 +6 a_1 -3)-\sqrt{\left(\frac{3 (a_1 ^2 +6 a_1 -3)}{4} \right)^2 -12 a_1 ^3 +324}.$$ 
	\end{enumerate}	 
	Here the positive square root is taken.
	\label{estimate-1}
\end{lemma}
\begin{proof} First note that $a_1 >3$ whenever $a_2 ^2 < 3 a_1 a_3$.
	\begin{enumerate}
		\item Squaring both sides of $|P(c)+\frac{a_2}{3a_3}| > \frac{2}{\sqrt{a_3}}$, we get $$(-9 a_1 a_2 a_3+2 a_2 ^3 + 9 a_2 a_3)^2 + 4 (3 a_1 a_3 -a_2 ^2)^3 > 4 (27)^2 a_3 ^3 .$$
		This gives that $4 a_2 ^4  - 3 a_2 ^2 a_3 (a_1 ^2 +6 a_1 -3) +12 a_1 ^3 a_3 ^2 > 324 a_3 ^2$, and this is equivalent to \begin{equation}
			\left(2 a_2 ^2 - \frac{3 a_3(a_1 ^2 +6 a_1 -3)}{4} \right)^2 > a_3 ^2 \left( \left( \frac{3 (a_1 ^2 +6 a_1 -3)}{4}\right)^2 -12 a_1 ^3 +324 \right)
			\label{aux-1}.\end{equation} 
		To show that the right-hand side expression is positive for all $a_1 >3$, consider $\psi(x)=\left( \frac{3 (x ^2 +6 x -3)}{4}\right)^2 -12 x ^3 +324 $. Note that $\psi'(x)=\frac{9}{4}(x^3 - 7x^2+15 x-9)=\frac{9}{4}(x-3)^2 (x-1)$, and it is positive for all $x >3$. This means that  $\psi$ is strictly increasing in $(3, \infty)$. Since $\psi(3)=324>0$, we have that $\psi(x)>0$ for all $x>3$.  It now follows from Equation (\ref{aux-1}) that $$|2 a_2 ^2 - \frac{3 a_3}{4} (a_1 ^2 +6 a_1 -3)| > a_3 \sqrt{\left(\frac{3 (a_1 ^2 +6 a_1 -3)}{4} \right)^2 -12 a_1 ^3 +324}.$$ It can be seen that the converse is also true.
		\item  Let $6  a_1 < \frac{3}{4} (a_1 ^2 +6 a_1 -3)-\sqrt{\left(\frac{3 (a_1 ^2 +6 a_1 -3)}{4} \right)^2 -12 a_1 ^3 +324}$. Then \\
		$$\sqrt{\left(\frac{3 (a_1 ^2 +6 a_1 -3)}{4} \right)^2 -12 a_1 ^3 +324} < \frac{3}{4} (a_1 ^2 +6 a_1 -3)- 6  a_1.$$ Squaring both sides, we get 
		$-12 a_1 ^3 +324 < -9 a_1 ( a_1 ^2 + 6a_1-3)+ 36 a_1 ^2 $.  This is nothing but $   a_1 ^3 -6 a_1 ^2 +9 a_1 -108 > 0 $. The converse can be found to be true by straight-forward calculations.
		
	\end{enumerate}
\end{proof}
\begin{lemma}
	Let $P(z)=a_1 z +a_2 z^2 +a_3 z^3$ be a reduced independence polynomial and $a_2 ^2<  3a_1 a_3$. If $a_1\geq 7$  then $\mathcal{J}(P)$ is totally disconnected.
	\label{biggerequaltoseven}\end{lemma}
\begin{proof}
	If $a_1\geq 7$ then $a_1 ^3-6 a_1 ^2 +9 a_1 -108 > 0$. To see this, note that   the function $\phi(x)= x^3 -6 x^2 +9 x -108$ is strictly increasing in $(3, \infty)$ (its derivative is $3(x-1)(x-3)$) and  $\phi(7)>0$). It now follows from Lemma~\ref{estimate-1}(2) that $6  a_1 < \frac{3}{4} (a_1 ^2 +6 a_1 -3)-\sqrt{\left(\frac{3 (a_1 ^2 +6 a_1 -3)}{4} \right)^2 -12 a_1 ^3 +324}.$ This is equivalent to $3 a_1 a_3 < \frac{a_3}{2} \left( \frac{3}{4} (a_1 ^2 +6 a_1 -3)-\sqrt{\left(\frac{3 (a_1 ^2 +6 a_1 -3)}{4} \right)^2 -12 a_1 ^3 +324} \right).$
	Since $a_2 ^2 < 3a_1 a_3$, 
	
	$a_2 ^2 	< \frac{a_3}{2} \left( \frac{3}{4} (a_1 ^2 +6 a_1 -3)-\sqrt{\left(\frac{3 (a_1 ^2 +6 a_1 -3)}{4} \right)^2 -12 a_1 ^3 +324} \right).$ In other words, 
	$2a_2 ^2 - \frac{3 a_3 (a_1 ^2 +6 a_1 -3)}{4} < - a_3 \sqrt{\left(\frac{3 (a_1 ^2 +6 a_1 -3)}{4} \right)^2 -12 a_1 ^3 +324}  $. Since the right-hand side is   negative, we have 	$\left |2a_2 ^2 - \frac{3 a_3 (a_1 ^2 +6 a_1 -3)}{4} \right | >a_3 \sqrt{\left(\frac{3 (a_1 ^2 +6 a_1 -3)}{4} \right)^2 -12 a_1 ^3 +324}  $.
	Now, using Lemma~\ref{estimate-1}(1), we have $|P(c)+\frac{a_2}{3a_3}| > \frac{2}{\sqrt{a_3}}$ where $c$ is the critical point of $P$ with positive imaginary part. Therefore, the image of the critical point $c$ is in the complement of the critical disk. It follows from Lemma~\ref{totally-disconnected}(2) that  the Julia set of $P$ is totally disconnected. 	\end{proof}

\begin{proof}[Proof of Theorem~\ref{bicritically non-real}]
	Since $a_2 ^2 < 3a_1 a_3$, it follows from Theorem~\ref{IA-JS} and  Remark~\ref{IA-JS-rem}(1) that $\mathcal{A}(G) =\mathcal{J}(P_G)$. Denoting $P_G$ by $P$, as stated earlier, it now becomes enough to show that $ \mathcal{J}(P )$ is totally disconnected.
	\par 	The proof follows from Lemma~\ref{biggerequaltoseven} for $a_1 \geq 7$. In order to  prove it for $a_1 \leq 6$, let $c$ denote the critical point of $P$ with positive imaginary part.
	
	For $a_1 =4$, the possible values of $( a_2,a_3)$ are $( 3,1),( 4,1), ( 5,2)$ by Lemma~\ref{basic-independentpoly}(1). The condition $a_2 ^2 < 3 a_1 a_3$ gives that  $(a_2,a_3)=(3,1)$,  and the reduced independence polynomial is $ 4z + 3z^2 + z^3$. The critical point of $P $ with positive imaginary part is $ c= -1+\frac{i}{\sqrt{3}}$ and the critical disk is $D_P=\{z:|z+1| \leq 2\}$. Note that $P(c )= -2 + \frac{i 2\sqrt{3}}{9}$ and the real part of $P^2 (c )  $ is $\frac{-32}{9}$. This gives that the forward orbit of the critical point $c$ intersects the complement of the critical disk. Therefore, the Julia set is totally disconnected by Lemma~\ref{totally-disconnected}(2).
	\par 
	For $a_1 =5$, if the Julia set of $P$ is connected then  $|P(c)+ \frac{a_2}{3a_3}| \leq \frac{2}{\sqrt{a_3}}$ by Lemma~\ref{connected-Juliaset}(1) and Lemma~\ref{filled-Julia-estimate}. It follows from Lemma~\ref{estimate-1}(1) that $ \frac{39 - \sqrt{345}}{2}a_3 \leq  a_2 ^2 \leq \frac{39 +\sqrt{345}}{2}a_3$, i.e.,  $ 10.21 a_3 \leq  a_2 ^2 \leq 28.78 a_3$.  However, we have a tighter upper bound for $a_2 ^2$ using the assumption that $a_2 ^2 <  3 a_1 a_3$, and we use this instead of $28.78 a_3$. Thus, if $\mathcal{J}(P)$ is connected then \begin{equation}
		10.21 a_3 \leq  a_2 ^2  < 15 a_3.
		\label{ineqa_1=5}
	\end{equation} 	
	It follows from Lemma~\ref{basic-independentpoly}(2) that $ 3 \leq a_2 \leq 8$ and, if $a_2 =8$ then $a_3=4$. The latter is not possible as $a_2 ^2 < 3 a_1 a_3$. Therefore,$$3 \leq a_2 \leq 7.$$
	It follows from the left-hand inequality of Inequation~(\ref{ineqa_1=5}) that if $a_3 \geq 5$ then $a_2 \geq 8$. It also follows from Inequation~(\ref{ineqa_1=5}) that $a_3 \neq 1$. Thus,   $2 \leq a_3 \leq 4 $. Putting all these values of $a_3$ in Inequation~(\ref{ineqa_1=5}), it is found that    $(a_2, a_3) \in \{(5,2),(6,3), (7,4)\}$.
	
	\par  For $(a_2, a_3)= (5,2)$, the critical disk of $P(z)=5z+5z^2+2z^3$ is $\{z: |z+\frac{5}{6}| \leq \sqrt{2} \}$ and the real part of $P^2 (c)$ is $-4.55 $.
	\par  For $(a_2, a_3)= (6,3)$, the critical disk of  $P(z)=5z+6z^2+3z^3$ is $ \{z: |z+\frac{2}{3}| \leq \frac{2}{\sqrt{3}} \}$ whereas the real part of $P^2 (c) $ is $-4.15$. 
	\par  Similarly, for $(a_2, a_3)= (7,4)$, the critical disk of $P(z)=5z+7z^2+4z^3$ is $ \{z: |z+\frac{7}{12}| \leq 1\}$ and the real part of $P^2 (c)$ is $-3.41$. 
	\par The values of $P^2 (c)$, the image of the critical value is taken to be rounded off to two decimal places. These  give that this image is outside the critical disk in each case, i.e., $(a_2, a_3) \in \{(5,2),(6,3), (7,4)\}$. See Table ~\ref{a_1=5} for detailed calculations.  Thus, the forward orbit of $c$ is unbounded. However, this is a contradiction to our assumption that the Julia set is connected (see Lemma~\ref{connected-Juliaset}(1)). Therefore the Julia set of $P$ is not connected.    It now follows from Lemma~\ref{totally-disconnected}(1) that the Julia set of $P$ is totally disconnected in all these cases.
	
	\par For $a_1 =6$, the proof is similar to the previous case.
	If the Julia set of $P$ is connected then it follows from Lemma~\ref{filled-Julia-estimate} and Lemma~\ref{estimate-1}(1) that $15.75 a_3 \leq a_2 ^2 \leq 36 a_3$. The assumption $a_2^2 < 3 a_1 a_3$ gives that $a_2 ^2 < 18 a_3$. Thus, if $\mathcal{J}(P)$ is connected then \begin{equation}
		15.75 a_3  \leq a_2 ^2 < 18 a_3.
		\label{ineqa_1=6}
	\end{equation}
	It follows from Lemma~\ref{basic-independentpoly}(3) that $3 \leq a_2 \leq 12$, and if $a_2 =12$ then $a_3 =8$. The latter is not possible as $a_2 ^2 < 3 a_1 a_3$. Therefore $$ 3 \leq a_2  \leq 11.$$
	
	It follows  from the first inequation of  Inequation~\ref{ineqa_1=6} that if $a_3 \geq 9 $ then   $a_2 \geq 12$. However, this does not agree with $3 \leq a_2 \leq 11$. Now if $a_3 =2$ or $8$ then there is no positive integer value of $a_2$ satisfying Inequation~\ref{ineqa_1=6}.	Thus, $a_3 \in \{1, 3, 4, 5, 6, 7 \}$. By using Inequation~\ref{ineqa_1=6}, we have $(a_2, a_3) \in \{(4,1), (7,3), (8,4), (9,5), (10,6), (11, 7)\}$. In each of these cases, as evident from Table~\ref{a_1=6}, $P^2 (c)$ is outside the critical disk of $P$. Therefore, the Julia set is not connected.  It now follows from Lemma~\ref{totally-disconnected}(1) that the Julia set of $P$ is totally disconnected in all these cases.
	
	\par In the Tables~\ref{a_1=5}, \ref{a_1=6}, $D_r (a)$ denotes the closed disk centered at $a$ and with radius $r$.

	\renewcommand{\arraystretch}{1.4} 
	\begin{longtable}{|m{.5cm}|m{0.5cm}|m{2cm}|m{2.5cm}|m{2.3cm}|m{1.5cm}|}
		\hline
		\textbf{$a_3$} & \textbf{$a_2$} & \centering \textbf{$c $} & \centering \textbf{$P(c)$} & \centering \textbf{$P^2(c)$} &Critical disk \\ 
		\hline
		2 & 5 & $-\frac{5}{6}+\frac{\sqrt{5}}{6}i$ & $-\frac{50}{27}+\frac{5\sqrt{5}}{54}i$ & $-4.55+1.44i$ & $D_{\sqrt{2}}(-\frac{5}{6})$ \\
		\hline
		3 & 6 & $-\frac{2}{3}+\frac{1}{3}i$ & $-\frac{14}{9}+\frac{2}{9}i$ & $-4.16+1.77i$ & $D_\frac{2}{\sqrt{3}}(-\frac{2}{3})$ \\
		\hline
		4 & 7 &$-\frac{7}{12}+\frac{\sqrt{11}}{12}i$ & $-\frac{287}{216}+\frac{11\sqrt{11}}{216}i$ & $-3.41+1.26i$ & $D_1(-\frac{7}{12})$ \\
		\hline
		
		\caption{\centering The value of $P^2 (c)$ for $a_1 =5$ when $a_2 ^2 < 3 a_1 a_3$.}
		\label{a_1=5}
	\end{longtable}
	
	\begin{longtable}{|m{0.4cm}|m{0.4cm}|m{1.7cm}|m{2.3cm}|m{2.5cm}|m{1.5cm}|}
		\hline
		\textbf{$a_3$} & \textbf{$a_2$} & \centering \textbf{$c $} & \centering $P(c )$ & \centering \textbf{$P^2(c)$} &  Critical disk  \\ 
		\hline
		1 & 4 & $-\frac{4}{3}+\frac{\sqrt{2}}{3}i$ & $-\frac{88}{27}+\frac{4\sqrt{2}}{27}i$ & $-11.43+2.46i$ & $D_2(-\frac{4}{3})$ \\
		\hline
		3 & 7 & $-\frac{7}{9}+\frac{\sqrt{5}}{9}i$ & $-\frac{448}{243}+\frac{10\sqrt{5}}{243}i$ & $-5.99+0.99i$ & $D_\frac{2}{\sqrt{3}}(-\frac{7}{9})$ \\
		\hline
		4 & 8 &$-\frac{2}{3}+\frac{\sqrt{2}}{6}i$ & $-\frac{44}{27}+\frac{2\sqrt{2}}{27}i$ & $-5.72+1.23i$ & $D_1(-\frac{2}{3})$ \\
		\hline
		5 & 9 & $-\frac{3}{5}+\frac{1}{5}i$ & $-\frac{972}{675}+\frac{2}{25}i$ & $-4.83+0.89i$ & $D_\frac{2}{\sqrt{5}}(-\frac{3}{5})$ \\
		\hline
		6 & 10 & $-\frac{5}{9}+\frac{\sqrt{2}}{9}i$ & $-\frac{310}{243}+\frac{8\sqrt{2}}{243}i$ & $-3.81+0.45i$ & $D_\frac{2}{\sqrt{6}}(-\frac{5}{9})$ \\
		\hline
		7 & 11 & $-\frac{11}{21}+\frac{\sqrt{5}}{21}i$ & $-\frac{1496}{1323}+\frac{10\sqrt{5}}{1323}i$ & $-2.84+0.13i$ & $D_\frac{2}{\sqrt{7}}(-\frac{11}{21})$ \\
		\hline
		
		\caption{\centering The value of $P^2 (c)$ for $a_1 =6$ when $a_2 ^2 < 3 a_1 a_3$.}
		\label{a_1=6}	
	\end{longtable}
	
\end{proof}
\begin{Remark}
	If $a_2 ^2<  3a_1 a_3$ then  $P$ has two distinct non-real critical points. If  $a_2 ^2< 3  a_1 a_3$ then $a_2 ^2  -4 a_3 (a_1 -1) <a_3 (4-a_1)$. Since $a_1 \geq 4$ (see Lemma~\ref{basic-obs}), we have $a_2 ^2  <4 a_3  a_1 -4a_3  $. This actually means that there is no non-zero real fixed point of $P$.
\end{Remark}
\subsection{Proof of Theorem~\ref{unicritical} }
A reduced independence polynomial $a_1 z +a_2 z^2 +a_3 z^3$ has a single critical point if and only if $a_2 ^2 = 3 a_1 a_3$. The critical point is $\frac{-a_2}{3 a_3}$. We determine all the possibilities of independence attractors in this case through the proof of Theorem~\ref{unicritical}.

\begin{proof}[Proof of Theorem~\ref{unicritical}]
	First we show that $a_1 \neq 4$.
	If $a_1 =4$, then $a_2 <6$, and it follows from Lemma~\ref{basic-independentpoly}(1) that $(a_2, a_3) \in \{(3,1), (4,1), (5,2)\}$. However, $a_2 ^2 =12 a_3$ is no longer satisfied by any of these possibilities.   
	\par If 
	$a_1=3$ then it follows from Lemma~\ref{basic-obs} that  $P(z)=3z+3z^2+z^3$ and therefore $I_G(z)=(1+z)^3$. This gives that $\mathcal{A}(G)=\{-1\}$ and $G$ is the empty graph on three vertices. 
	
	\par Finally, let $a_1 >4$.
	Since $a_2 ^2 -4   a_3 (a_1 -1) <0$, the two non-zero fixed points of $P$ are non-real. In other words, $P$ has no fixed point on the negative real axis.  By Lemma~\ref{Juliaset-prop}(3), the Julia set of $P$ is totally disconnected. We are done since the Julia set is the independence attractor of the graph by Theorem~\ref{IA-JS} and Remark~\ref{IA-JS-rem}(1). 
	
\end{proof}
\begin{Remark}
	\begin{enumerate}
		\item There is only one graph for $a_1 =3$.	
		\item 
		For $a_1 >4$, the fixed points are $\alpha$ and $\overline{\alpha}$ where $\alpha =\frac{-a_2+i \sqrt{4 a_3 (a_1 -1)-a_2 ^2}}{2 a_3}$. Using  $a_2 ^2 = 3 a_1 a_3$, we get $\alpha^2 =\frac{a_1 +2-i \sqrt{3a_1 (a_1-4)}}{2a_3}$ and $P'(\alpha)=\frac{-a_1 +6-i \sqrt{3a_1 (a_1-4)}}{2}$. Note that $|P'(\alpha)|^2=(a_1 -3)^2$.  As  $a_1 >4, |P'(\alpha)|>1$. Therefore $|P'(\overline{\alpha})|>1$.
		Hence all the finite fixed points are repelling. 
		\item  For $a_1 >9$, Lemma~\ref{Buff-2003} gives that the Julia set of $P$ is disconnected. But Theorem~\ref{unicritical} is a statement about $a_1 >4$ and it gives that the Julia set of $P$ is totally disconnected.
		
	\end{enumerate}
	\label{unicritical-fixedpoints}
\end{Remark}
\subsection{Proof of Theorem~\ref{bicritically real-two non-real fixed points or one fixed point} }
A cubic reduced independence polynomial $P(z)=a_1 z+a_2 z^2 +a_3 z^3$ has two distinct real critical points whenever $a_2 ^2> 3a_1 a_3$.
The  critical points are 
\begin{equation}
	c_1=\frac{-a_2 -\sqrt{a_2 ^2 -3a_1 a_3}}{3a_3}~\mbox{ and}~ c_2 =\frac{-a_2 +\sqrt{a_2 ^2 -3a_1 a_3}}{3a_3}.
	\label{two-criticalpoints}
\end{equation} Note for later use that  $c_1 < c_2$ and for $i =1,2$, 
\begin{equation}
	P(c_i) = \frac{6 a_1 a_3 c_i -a_1 a_2 -2 a_2 ^2 c_i}{9 a_3}. 
	\label{criticalvalue-bicritical-real}
\end{equation} 
We also need another observation:  \begin{equation}  P'(z)=3a_3 (z-c_1) (z-c_2). \label{P-derivative}\end{equation} 

Recall that the non-zero fixed points of $P$ are  $$\delta_1 =\frac{-a_2 - \sqrt{a_2 ^2 -4a_3 (a_1 -1)}} {2a_3} ~\mbox{and }~ \delta_2 =\frac{-a_2 + \sqrt{a_2 ^2 -4a_3 (a_1 -1)}} {2a_3}. $$
There are three situations depending on the nature of non-zero fixed points of $P$ that require different treatments.
\begin{itemize}
	\item [(A)]$a_2 ^2 <4 a_3(a_1 -1)$: Two distinct non-real  fixed points (with the same real part).
	\item [(B)] $a_2 ^2= 4 a_3(a_1 -1)$: One fixed point and it is $-\frac{a_2}{2a_3}$.
	
	\item [(C)]$a_2 ^2 > 4 a_3(a_1 -1)$: Two distinct real fixed points, and both are negative.
	
\end{itemize} 
\begin{Observation}
	
	\begin{enumerate}
		\item 	Since $P$ has two real critical points $c_1 < c_2$, we have  $P'(x)=3a_3(x-c_1)(x-c_2)$ for $ x\in\mathbb{R}$. This shows that $P$ is strictly increasing in $(-\infty,c_1) \cup(c_2,\infty)$ and is strictly decreasing in $(c_1,c_2)$.
		\item  If there are two real non-zero fixed points of $P$,  $\delta_1 < \delta_2 < 0$ then $P(x)-x=a_3x(x-\delta_1)(x-\delta_2)$ for $ x\in\mathbb{R}$. This gives that $P(x)<x$ for $x \in (-\infty, \delta_1) \cup ( \delta_2, 0 )$ and $P(x)>x$ for $x \in (\delta_1, \delta_2) \cup (0, \infty)$.
		\item Let $P$ have two real non-zero fixed points,  $\delta_1 < \delta_2 < 0$.  	Since $P(x)<x$ for $x \in (-\infty, \delta_1)$, the sequence $\{P^n (x)\}_{n>0}$ is strictly decreasing and $\lim\limits_{n \to \infty} P^n (x)=-\infty$.
	\end{enumerate}
	\label{obs 3.1}
\end{Observation}
We now present the proof of Theorem \ref{bicritically real-two non-real fixed points or one fixed point}.
\begin{proof}[Proof of Theorem~\ref{bicritically real-two non-real fixed points or one fixed point}(1)] ({Two distinct non-real fixed points: $a_2 ^2 <4a_3 (a_1 -1)$})
	
		The  hypothesis   $a_2 ^2 <4a_1 a_3-4a_3$ implies that $P$ does not have any real non-zero fixed point. The  Julia set of $P$ is totally disconnected by Lemma~\ref{Juliaset-prop}(3). The independence attractor $\mathcal{A}(G)$ is the Julia set of $P$ by Theorem~\ref{IA-JS} and Remark~\ref{IA-JS-rem}(1), and we are done.
		
	\end{proof}
	\begin{Remark}
		If $a_2 ^2> 3a_1 a_3$ and $a_2 ^2 <4a_1 a_3-4a_3$ then $a_1 \geq 5$.
	\end{Remark}


	\begin{proof}[Proof of Theorem~\ref{bicritically real-two non-real fixed points or one fixed point}(2)]({One non-zero fixed point: $a_2 ^2 =4a_3 (a_1 -1)$})

		The polynomial $P$ has a single non-zero finite fixed point if  $a_2 ^2 =4a_1 a_3-4a_3$. 
		Let the fixed point $\frac{-a_2}{2 a_3}$ be denoted by $\delta$. Note that $\delta < c_1$.

		That $\delta$ is the parabolic fixed point of $P$, i.e., $P'(\delta)=1$ can be seen using $a_2 ^2 =4 a_3 (a_1-1)$. Since $\delta$ is the only non-zero fixed point of $P$ and that is real, we have $P(x)-x=a_3 x (x-\delta)^2$ for all real $x$. This gives that for all $x <0, x \neq \delta$, $P(x)-x  < 0$. In particular, $P(x) < x$ for all $x \in (\delta,  c_1]$. It is also clear from Observation~\ref{obs 3.1}(1) that $P$ is strictly increasing in $ (\delta, c_1)$. This gives that   the sequence $\{P^n (x)\}_{n>0}$ is strictly decreasing and bounded below by $\delta$, and therefore $\lim \limits_{n \to \infty} P^{n} (x) =\delta$ for all $x \in (\delta, c_1]$.
		Since $P$ is strictly increasing and $P(x)<x$ in $(-\infty, \delta)$, we have  
		\begin{equation}
			\lim_{n \to \infty} P^n (x) = -\infty 
			~\mbox{for all }~x < \delta .
			\label{basin-infinity}
		\end{equation} 
		Now, we look at $P(c_2)$ first, in order to understand its forward orbit. 
		Using Equation~\ref{criticalvalue-bicritical-real} and $a_2 ^2 =4a_3 (a_1 -1)$, it can be seen that  $$P(c_2)-\delta = \frac{-4 (a_1 -4)(-a_2 +\sqrt{a_3 (a_1 -4)})+27 a_2-6 a_1 a_2}{54 a_3}.$$
		Now, using $a_2 =2\sqrt{a_3(a_1 -1)}$  and  simplifying the expression, we have
		\begin{equation}
			P(c_2)-\delta = \frac{(11-2a_1)(\sqrt{a_1 -1})- 2 (a_1 -4)\sqrt{a_1 -4}}{27\sqrt{a_3}}. 
			\label{alpha-gamma}
		\end{equation}
		\begin{center}
			\begin{figure}[h!]
				\begin{center}
					\begin{subfigure}{.5\textwidth}
						\includegraphics[width=1.0\linewidth]{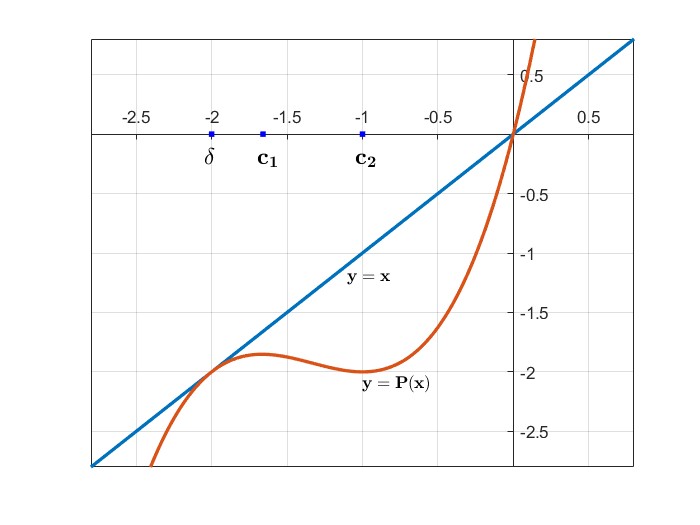}
						\caption{Graph of $P(x)=x^3+4x^2+5x$}
					\end{subfigure}%
					\begin{subfigure}{.5\textwidth}
						\centering
						\includegraphics[width=1.0\linewidth]{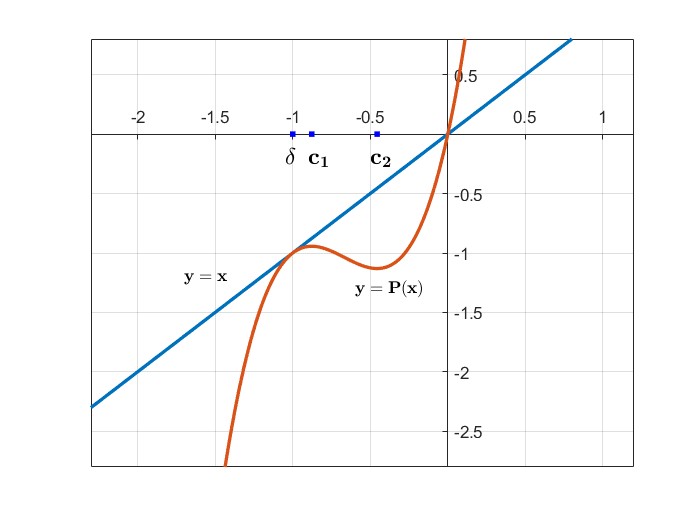}
						\caption{Graph of $P(x)=5x^3+10x^2+6x$}
					\end{subfigure}
				\end{center}
				\caption{Graph of $P (x),  x \in \mathbb{R}$ when it has two real critical points and a single real fixed point: (a) is an example of $P$ where the critical value corresponding to the bigger critical point is equal to the real fixed point and (b) is an example of $P$ where the critical value corresponding to the bigger critical point is less than the real fixed point.}	
				\label{bicriticalreal-singlefixedpoint}
			\end{figure}	
		\end{center}
		In this case, $a_2 ^2 > 3a_1 a_3$ gives that $a_1 \geq 5$.  
		If  $a_1 =5$ then   $P(c_2)=\delta$ (see Figure~\ref{bicriticalreal-singlefixedpoint}(a) where $a_3=1$). Consequently, the interval $[\delta, 0]$ is mapped onto itself and $P(x)<x$ for all points in $(\delta,0)$. 
		For each point $x \in (\delta,0)$, either $P^n (x)=\delta$ for some $n$ or $\{P^n (x)\}_{n >0}$ is strictly decreasing and bounded below by $\delta$.
		There are infinitely many points $x$ including $c_2$ satisfying the first possibility, and all these points are in $(c_2,0)$.  For each $x$ satisfying the second possibility, $P^n (x)$   must converge and the limit point should be a fixed point in $[\delta, 0)$. As $\delta$ is the only fixed point  in this interval,  we have $\lim\limits_{n \to \infty}P^n (x)=\delta$ for all $x \in (\delta, 0)$. 
		As both the critical points remain bounded under the iteration of $P$, the Julia set of $P$ is connected by Lemma~\ref{connected-Juliaset}(1).
		
		For $a_1 >5, $ we have $P(c_2) - \delta <0$ (see Figure~\ref{bicriticalreal-singlefixedpoint}(b) where $a_1=6$ and $ a_3=5$). It follows from Equation~\ref{basin-infinity} that $\lim\limits_{n\to \infty} P^n (c_2) =-\infty$. The Julia set is disconnected by Lemma~\ref{connected-Juliaset}(1). It is not totally disconnected as there are at least two Fatou components namely the basin of $\infty$ and the parabolic domain corresponding to $\delta$. Since $\mathcal{A}(G)$ is the Julia set of $P$ in this case by Theorem~\ref{IA-JS} and Remark~\ref{IA-JS-rem}(1), we are done. \end{proof}
	
	\begin{Remark}
		\begin{enumerate}
			
			\item Combining $a_2^2>3a_1a_3$ with $a_2^2=4a_3(a_1-1)$ gives $a_1>4$.
			For $a_1=5$ we have $a_2^2=16a_3$, and $(a_2,a_3)=(8,4)$ occurs:
			the graph $2K_2\cup K_1$, consisting of two disjoint edges and an isolated
			vertex, has reduced independence polynomial $5z+8z^2+4z^3$.
			For $a_1=6$ we have $a_2^2=20a_3$, and $(a_2,a_3)=(10,5)$ occurs.
			Let $C$ be the graph obtained by identifying a pendant vertex of the path
			$P_3$ with a vertex of a triangle. Then $C\cup K_1$ has reduced independence
			polynomial $6z+10z^2+5z^3$.
			\item For $a_1=5$, it can be proved that the parabolic domain corresponding to $\delta$ is not completely invariant leading to infinitely many components of the Fatou set. Therefore, the Julia set of its reduced independence polynomial is far from being  a Jordan curve.
		\end{enumerate}
	\end{Remark}
	\subsection{Proof of Theorem~\ref{bicritically real-two real fixed points}}
		
		If $a_{2}^2 >4a_3 (a_1 -1)$ then the non-zero real fixed points of $P$ are
		\begin{equation}
			\delta_1=\frac{-a_2 -\sqrt{a_2 ^2 -4a_3(a_1-1)}}{2a_3}~\mbox{ and }~\delta_2=\frac{-a_2 +\sqrt{a_2 ^2 -4a_3(a_1-1)}}{2a_3}.
			\label{two-fixedpoints}
		\end{equation} Clearly, $\delta_1<\delta_2 <0$. Since $a_1 \geq 4$ (see Lemma~\ref{basic-obs}), we have $4a_3(a_1-1)\geq 3a_1a_3$. 
		The two assumptions  $a_2 ^2> 3a_1 a_3$ and $a_2 ^2 >4a_1 a_3-4a_3$ reduces to  $$a_2 ^2 >4a_3(a_1-1).$$
		Using this inequality and noting that the critical points are $c_1=\frac{-a_2 -\sqrt{a_2 ^2 -3a_1 a_3}}{3a_3}$ and $c_2 =\frac{-a_2 +\sqrt{a_2 ^2 -3a_1 a_3}}{3a_3}$,  we have \begin{equation}
			\delta_1 < c_1 < c_2 ~\mbox{and} ~\delta_1 <\delta_2 <c_2.
			\label{critical-fixed-points}
		\end{equation} 
		
		The non-zero zeros of $P$ are $$\frac{-a_2 \pm \sqrt{a_2 ^2 -4 a_1 a_3}}{2a_3}.$$ 
		There are three cases depending on the value of $a_2 ^2 -4 a_1 a_3$, i.e., the nature of non-zero zeros of $P$.
		\begin{itemize}
			\item[(I)]  $a_2 ^2 < 4a_1 a_3$: Two distinct non-real zeros of $P$ (with the same real part).
			\item[(II)] $a_2 ^2=4a_1 a_3$: One non-zero zero of $P$.
			\item[(III)] $a_2 ^2> 4a_1 a_3$: Two distinct real non-zero zeros of $P$. 
		\end{itemize}
		
		\begin{center}
			\begin{figure}[h!]
				
				\begin{subfigure}{.5\textwidth}
					\centering
					\includegraphics[width=1\linewidth]{ 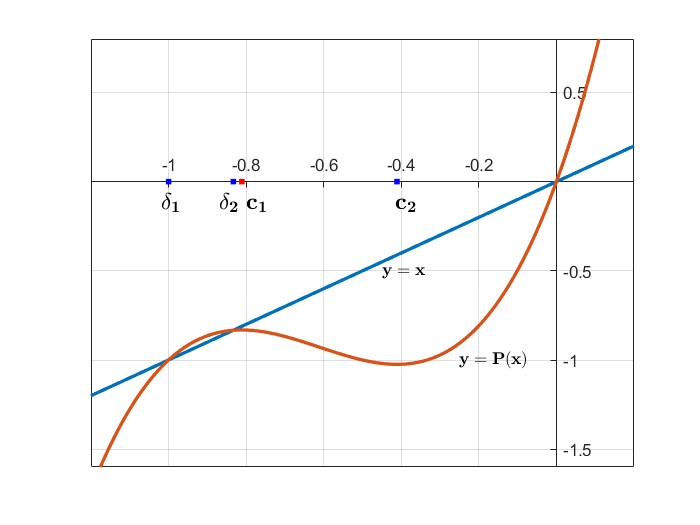}
					\caption{Graph of $P(x)=6x^3+11x^2+6x$.}
					\label{bicriticallyreal-tworealfixedpoints-case1}
				\end{subfigure}%
				\begin{subfigure}{.5\textwidth}
					\centering
					\includegraphics[width=1\linewidth]{ 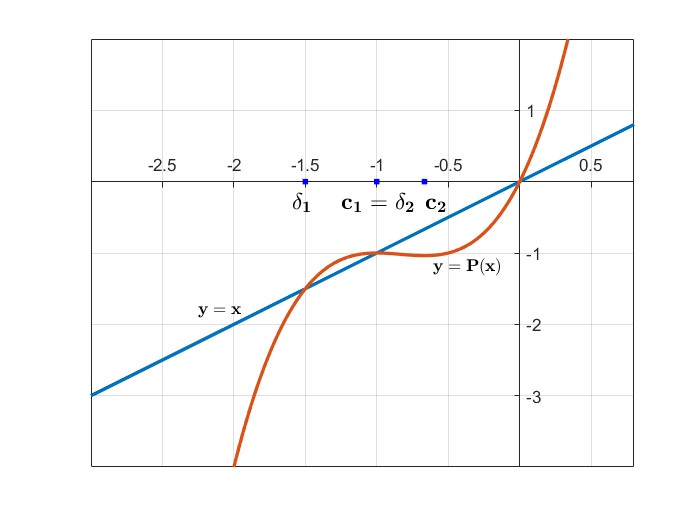}
					\caption{Graph of $P(x)=2x^3+5x^2+4x$.}
					\label{bicriticallyreal-tworealfixedpoints-case2}
				\end{subfigure}\\[1ex]
				\centering
				\begin{subfigure}{0.5\textwidth}
					\centering
					\includegraphics[width=1\linewidth]{ 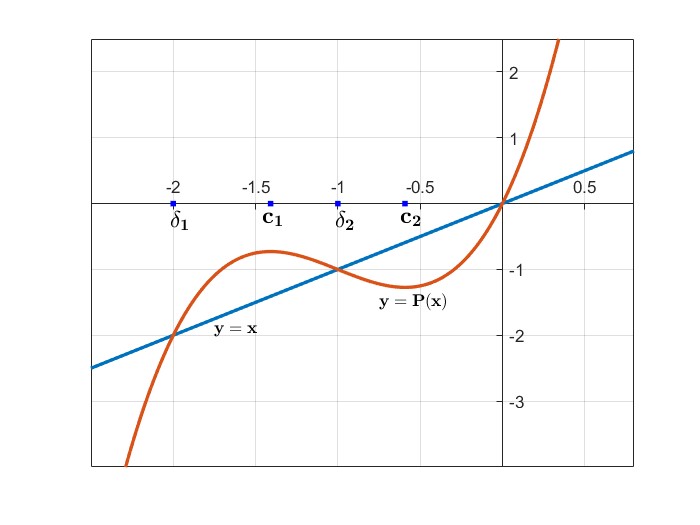}
					\caption{Graph of $P(x)=2x^3+6x^2+5x$.}
					\label{bicriticallyreal-tworealfixedpoints-case3}
				\end{subfigure}
				\caption{(a)   $P(c_2)< \delta_1$ when $4a_3(a_1-1)<a_2^2<\frac{a_3(2a_1-3)^2}{a_1-2}$; (b) $P(c_2)> \delta_1$ when $a_1=4$ and $4a_3(a_1-1)<a_2^2<4a_1a_3$; (c) $P(c_2)> \delta_1$ when $a_2^2=\frac{4 a_3( a_1-2)^2}{a_1-3}$. }
			\end{figure}
		\end{center}
		
		\par  \noindent \underline{\textbf{Case I - $a_2^2 < 4a_1a_3$  - Two distinct non-real zeros of P}}:\

		Before analyzing the sub cases of $a_2^2 < 4a_1a_3$, we state and prove a lemma that demonstrates how the values of the smaller fixed point $\delta_1$ and the critical value corresponding to the larger critical point $c_2$ determine the connectedness of the Julia set.   
		\begin{lemma} \label{$c_2 escape$}
			Let  $P(z)=a_1 z +a_2 z^2 +a_3 z^3$ be the reduced independence polynomial with $ 4a_3 (a_1 -1) < a_2 ^2 < 4 a_1 a_3 $. Then the following statements hold.
			\begin{enumerate}
				\item The Julia set of $P$ is connected if and only if $\delta_1 \leq P(c_2)$.
				\item For $a_1 \geq 9$, the Julia set is disconnected and the larger critical point $c_2$ escapes. \end{enumerate}
			\label{bicriticalreal-tworealfixedpoint-twononrealzero-lemma}
		\end{lemma}
		\begin{proof}
			\begin{enumerate}
				\item	Since $a_2 ^2 < 4 a_1 a_3$, $P$ has no negative zero. It follows that $P(c_1)<0$. If $\delta_1 \leq P(c_2)$, then by Observation \ref{obs 3.1}(1), $P([\delta_1,0])=[\delta_1,0]$ and $P(c_1)> P(c_2) \geq \delta_1$(See Figure~\ref{bicriticallyreal-tworealfixedpoints-case2},\ref{bicriticallyreal-tworealfixedpoints-case3}). This shows that the forward orbits of both the critical points   $ c_1 $ and $ c_2 $ are bounded. Therefore, the Julia set of $P$ is connected by Lemma~\ref{connected-Juliaset}(1). On the other hand, if $\delta_1 > P(c_2)$ then $c_2$ escapes by Observation~\ref{obs 3.1}(3). Consequently, the Julia set of $P$ is disconnected by Lemma~\ref{connected-Juliaset}(1). 
				\item  Lemma \ref{Buff-2003} gives that the  Julia set of $P$ is disconnected for $a_1 > 9$, and for $a_1 =9$, it is a line segment whenever it is connected. Theorem 2.2 of \cite{line} states that the Julia set of a reduced independence polynomial (independence fractal) of a graph with independence number $3$  is a line segment if and only if the reduced independence polynomial   is $9z +6kz^2 +k^2 z^3 $ for some $k=1,2,3,4,5$. As $a_2 ^2 < 4a_1 a_3$ is not satisfied, the Julia set of $P$ is not a line segment. Thus the  Julia set of $P$ is disconnected for $a_1 \geq 9$. By the first part of this lemma,  $P(c_2)<\delta_1$ for $a_1 \geq 9$. By Observation~\ref{obs 3.1}(3), we conclude that $c_2$ escapes when $a_1 \geq 9$. 
		\end{enumerate}\end{proof}

		By Lemma~\ref{bicriticalreal-tworealfixedpoint-twononrealzero-lemma}, the location of $\delta_1$ and $ P(c_2)$ with respect to each other completely determines the connectedness of the Julia set.
		The nature of $\delta_2$ and its position with respect to $c_1$  give rise to different situations, that call for different treatments. We state the following for using later.
		\begin{lemma} Let $P(z)=a_1z + a_2 z^2 +a_3 z^3$ be a reduced independence polynomial. If $a_1 >4$ and $ 4a_3 (a_1 -1) < a_2 ^2 < 4 a_1 a_3 $  then the smaller fixed point $\delta_1$ of $P$ is always repelling. The nature of the larger fixed point is given below.
			\begin{enumerate}
				\item If $a_2 ^2 < \frac{a_3(2a_1 -3 )^2}{a_1 -2}$ then $\delta_2$ is attracting and $\delta_2 < c_1$.
				\item If $a_2 ^2 = \frac{a_3(2a_1 -3 )^2}{a_1 -2}$ then $\delta_2$ is super-attracting and $\delta_2 = c_1$.
				\item 	If $\frac{a_3(2a_1 -3 )^2}{a_1 -2}< a_2 ^2  < \frac{4 a_3(a_1 -2)^2}{a_1 -3}$ then $\delta_2$ is attracting and $\delta_2 > c_1$.
				\item If $ a_2 ^2  = \frac{4 a_3(a_1 -2)^2}{a_1 -3}$ then $\delta_2$ is parabolic and $\delta_2 > c_1$.
				\item If $  a_2 ^2 > \frac{4 a_3(a_1 -2)^2}{a_1 -3}$ then  $\delta_2$ is repelling and $\delta_2 > c_1$.
			\end{enumerate}
			\label{five-cases}
		\end{lemma}
		\begin{proof} Since $P(\delta_i)=\delta_i$, we have  $P'(\delta_i)= 3-2a_1-a_2 \delta_i$ for $i=1,2$.
			\par  Putting the value of $\delta_1$, we get 
			$P'(\delta_1)=1+\frac{a_2 ^2-4a_3(a_1-1)+ a_2\sqrt{a_2^2-4a_3(a_1-1)}}{2a_3}$. As  $a_2 ^2 >4a_3(a_1-1)$, we get $P'(\delta_1)>1$. Therefore $\delta_1$ is a repelling fixed point.
			\par 
			Note that $P'(\delta_2)=1+2-2 a_1-a_2 (\frac{-a_2 +\beta}{2a_3})$ where $\beta =\sqrt{a_2 ^2 -4a_3(a_1-1)}$. Calculating further, we have $P'(\delta_2)=1+\frac{4a_3-4a_1 a_3 +a_2 ^2 -a_2 \beta}{2 a_3}=1-\frac{\beta(a_2 -\beta)}{2a_3}$. Multiplying $a_2 +\beta$ in the numerator as well as in the denominator of the second term, it is found that $P'(\delta_2)=1-\frac{\beta (a_2 ^2 -\beta^2)}{2a_3 (a_2 +\beta)}$. Therefore,  
			\begin{equation}
				P' (\delta_2) =1 - \Delta,~\mbox{ where}~\Delta = \frac{2 (a_1 -1) \beta}{a_2 + \beta}.
				\label{multiplier-bigger-fixedpoint}
			\end{equation} Clearly $\Delta >0$. Further,  calculations show that 
			
			\begin{equation}
				\Delta \in  \left\{ \begin{array}{lll}
					(0,1) & \mbox{for $4a_3 (a_1 -1) < a_2 ^2 < \frac{a_3(2a_1 -3 )^2}{a_1 -2} $}\\
					\{1\} & \mbox{for $a_2 ^2 = \frac{a_3(2a_1 -3 )^2}{a_1 -2} $}\\
					(1,2) &\mbox{for $\frac{a_3(2a_1 -3 )^2}{a_1 -2}< a_2 ^2  < \frac{4a_3 (a_1 -2)^2}{a_1-3} $}\\
					\{2\} &\mbox{for $a_2 ^2 =\frac{4a_3 (a_1 -2)^2}{a_1-3} $}\\
					(2, \infty) &\mbox{for $ \frac{4a_3 (a_1 -2)^2}{a_1-3}  < a_2 ^2 < 4 a_1 a_3$}.	
				\end{array} \right.
				\label{Delta}	
			\end{equation}
			%
			
			To determine the location of $\delta_2$ with respect to $c_1$, first note that the inequality $a_2^2<\frac{a_3(2a_1-3)^2}{a_1-2}$ is equivalent to 
			$$2a_2^2-a_1a_2^2+4a_1^2a_3-12a_1a_3+9a_3>0.$$ Multiplying both sides by $9 a_3$ and then adding $a_2 ^4$ on both sides, we get 
			$ (6a_1a_3-9a_3-a_2^2)^2>a_2^2(a_2^2-3a_1a_3).$ It follows from $a_1>4$ and $a_2^2<4a_1a_3$  that $6a_1a_3-9a_3-a_2^2=4a_1 a_3 -a_2 ^2 +2a_1 a_3 -9a_3>0.$ Since $a_2^2 -3a_1a_3 >0$, taking positive square root, we have  $6a_1a_3-9a_3-a_2^2>a_2\sqrt{a_2^2-3a_1a_3}$. Now multiplying $4$ and then  adding $9 a_2 ^2 -36 a_1 a_3$ on both sides, we have
			$$\left(a_2-2\sqrt{a_2^2-3a_1a_3}\right)^2>9\left(a_2^2-4a_1a_3+4a_3\right).$$  Taking positive square root, we get that $\left(a_2-2\sqrt{a_2^2-3a_1a_3}\right)>3\sqrt{a_2^2-4a_1a_3+4a_3}.$ This shows that (see Equations ~\ref{two-criticalpoints} and \ref{two-fixedpoints})
			
			\begin{equation}
				\delta_2    \left\{ \begin{array}{lll}
					< c_1 & \mbox{for $4a_3 (a_1 -1) < a_2 ^2 < \frac{a_3(2a_1 -3 )^2}{a_1 -2} $}\\
					=c_1 & \mbox{for $a_2 ^2 = \frac{a_3(2a_1 -3 )^2}{a_1 -2} $}\\
					>c_1 &\mbox{for $  \frac{a_3(2a_1 -3 )^2}{a_1 -2}  < a_2 ^2 < 4 a_1 a_3$}.	
				\end{array} \right.
				\label{delta2-location}	
			\end{equation}

			\begin{enumerate}
				\item 
				For  $a_2 ^2 < \frac{a_3(2a_1 -3 )^2}{a_1 -2}$, it follows from Equation~\ref{Delta} that $0< \Delta <1$ (giving that $0< P' (\delta_2) <1$), and therefore $\delta_2$ is attracting. We have $\delta_2 < c_1$ by Equation~\ref{delta2-location}.  	 
				
				\item For $a_2^2=\frac{a_3(2a_1-3)^2}{a_1-2}$, $P'(\delta_2)=0$ by Equation~~\ref{Delta}. In other words, $\delta_2$ is a super-attracting fixed point of $P$. It follows from Equation ~\ref{delta2-location} that $\delta_2 =c_1$. 
				\item For  $\frac{a_3(2a_1 -3 )^2}{a_1 -2}< a_2 ^2  < \frac{4a_3 (a_1 -2)^2}{a_1-3} $, we have from Equation~\ref{Delta} that $P'(\delta_2) \in (-1,0)$ giving that $\delta_2$ is attracting. That $\delta_2 >c_1$ follows from Equation~\ref{delta2-location}.
				
				\item For $a_2^2=\frac{4a_3(a_1-2)^2}{a_1-3}$,  $P'(\delta_2)=-1$ by Equation~\ref{Delta} and therefore $\delta_2$ is parabolic. By Equation~\ref{delta2-location}, $\delta_2 >c_1$.
				\item    For $\frac{4a_3(a_1-2)^2}{a_1-3}<a_2^2<4a_1a_3$, we have $P'(\delta_2)<-1$ by Equation~\ref{Delta}. Therefore,  $\delta_2$ is a repelling fixed point of $P $ and  $\delta_2>c_1$ by Equation~\ref{delta2-location}.  
			\end{enumerate}
		\end{proof}
		The next lemma finds all possible reduced independence polynomials.
		\begin{lemma}Let $P(z)=a_1 z + a_2 z^2 +a_3 z^3$ be a reduced independence polynomial.
			Let $4 < a_1 \leq 8$ and $ 4a_3 (a_1 -1) < a_2 ^2 < 4 a_1 a_3 $. Then there are $21$  different possibilities for $P$ as given below. 
			\begin{enumerate}
				\item If $4a_3(a_1-1)<a_2^2<\frac{a_3(2a_1-3)^2}{a_1-2}$  then $ (a_1, a_2, a_3) =(k, 2k -1, k)$ for $k=6,7,8.$
				\item If $ a_2^2 =\frac{a_3(2a_1-3)^2}{a_1-2}$ then $(a_1, a_2, a_3) =(k, 2k -3, k-2)$ for $k=5,6,7,8$.
				\item If $\frac{a_3(2a_1-3)^2}{a_1-2}<a_2^2<\frac{4a_3(a_1-2)^2}{a_1-3}$ then  $(a_1, a_2, a_3) \in \{ (7,7,2), (7,14,8), (8,16,9) \}.$ 
				\item If $a_2^2=\frac{4a_3(a_1-2)^2}{a_1-3}$ then $$(a_1, a_2, a_3) \in \{(5,6,2), (6,8,3), (7,5,1),(7,10,4),(7,15,9), (8,12,5)\}.$$
				\item If  $\frac{4a_3(a_1-2)^2}{a_1-3}<a_2^2<4a_1a_3 $ then $$(a_1, a_2, a_3) \in \{(7,9,3), (8,11,4), (8,17,10), (8, 18,11), (8, 19, 12)\}.$$     	
			\end{enumerate}
			\label{bicriticalreal-twofixedpoints-twononrealzero}
		\end{lemma}	
		
		\begin{proof}
			\begin{enumerate}
				\item  Let $4a_3(a_1-1)<a_2^2<\frac{a_3(2a_1-3)^2}{a_1-2}$. Then there are four cases depending on the value of $a_1$, out of which  $a_1 =5$ is ruled out. 
				\par
				For $a_1=5$, we have
				\begin{equation} \label{eq a_1=5}
					16a_3<a_2^2<\frac{49}{3}a_3.
				\end{equation}  Note that $a_2 <9$ (see Lemma~\ref{basic-independentpoly}(2)) and $a_3 \leq 10$. If $a_3\in \{1,2,3,4\}$ then there is no possible value  for $a_2$   satisfying Inequation~\ref{eq a_1=5}. If $a_3 \geq 5$, the left- hand side inequality of ~\ref{eq a_1=5} implies that $a_2^2>80$, which is not possible.  
				\par 
				For $a_1=6$, we obtain
				\begin{equation} \label{eq a_1=6}
					20a_3<a_2^2<20.25a_3.
				\end{equation}  Note that $a_2 < 13$ (see Lemma~\ref{basic-independentpoly}(3)) and $a_3 \leq 20$. If  $ 1 \leq a_3 \leq 7, a_3 \neq 6$ then there is no possible value  for $a_2$  satisfying the Inequation~\ref{eq a_1=6}.
				If $a_3 \geq 8$, then the  left-hand side  inequality of \ref{eq a_1=6} implies that $a_2^2>160$, i.e., $a_2 \geq 13$, which is not true.  The only  remaining possibility is  $a_3=6$. That $a_2 =11$ follows from Inequation~\ref{eq a_1=6}, and  $(a_1, a_2, a_3)=(6,11,6)$ is the only possibility whenever $a_1=6$.
				\par 
				For $a_1=7$, we have
				\begin{equation} \label{eq a_1=7}
					24a_3<a_2^2<24.2 a_3
				\end{equation}  Note that $a_2 <17$ (see Lemma~\ref{basic-independentpoly}(4)) and $a_3 \leq 35$. If  $1 \leq a_3 \leq 11, a_3 \neq 7$ then there is no possible value  for $a_2$  satisfying  Inequation~\ref{eq a_1=7}. For $a_3 \geq 12$, the left inequality of   \ref{eq a_1=7} implies that $a_2^2>288$, i.e., $a_2 \geq 17$. But this is not possible.  Now, if $a_3=7$, then Inequation~\ref{eq a_1=7} gives $a_2=13$. Therefore $(a_1, a_2, a_3)=(7,13,7)$  is the only possibility for $a_1=7$. 
				\par 
				For $a_1=8$, we obtain
				\begin{equation} \label{eq a_1=8}
					28a_3<a_2^2<\frac{169}{6}a_3
				\end{equation}  Note that $a_2 < 22$ (see Lemma~\ref{basic-independentpoly}(5)) and $a_3 \leq 56$. If  $1 \leq a_3 \leq 15, a_3 \neq 8$ then there is no possible value  for $a_2$   satisfying Inequation~\ref{eq a_1=8}. 
				If $a_3 \geq 16$, the left inequality of  \ref{eq a_1=8} implies that $a_2^2>448$, i.e., $a_2 \geq 22$. This is however not possible.
				Thus  $a_3=8$, and Inequation~\ref{eq a_1=8} gives that  $a_2=15$. Therefore $(a_1, a_2, a_3)=(8,15,8)$  is the only possibility whenever $a_1=8$. 
				
				\item   If $a_2^2=\frac{a_3(2a_1-3)^2}{a_1-2}$ then $a_3 =(a_1 -2) m^2$  and $a_2 = (2a_1 -3)m$ for some positive integer $m$. However, $a_2  \leq \frac{a_1 (a_1 -1)}{2}$ gives that $m \leq \frac{a_1 (a_1 -1)}{2 (2a_1 -3)}$. For $a_1 \in \{5,6,7\}$, we have $m =1$ and therefore $a_2 = 2a_1 -3$ and $a_3 = a_1 -2$. Thus  $(a_1, a_2, a_3) =(k, 2k -3, k-2)$ for $k=5,6,7$. 
				
				For $a_1=8$, we have $m\in  \{1,2\}$. This leads to two possibilities:  $a_2 =13, a_3 =6$ or $a_2 =26, a_3 =24 $. But $a_2 =26$ is not possible by Lemma~\ref{basic-independentpoly}(5). Thus   the only possibility is $(a_1, a_2, a_3)=(8, 13,6)$  for $a_1=8$.
				\item Let $\frac{a_3(2a_1-3)^2}{a_1-2}<a_2^2<\frac{4a_3(a_1-2)^2}{a_1-3}.$ 
				Arguing similarly as in the proof of (1) of this lemma, we find that:
				\begin{enumerate}
					\item There is no reduced independence polynomial corresponding to $a_1 =5$.
					\item For $a_1 =6$, we must have $a_3 =7$ and $a_2 = 12$. This is  not possible by Lemma~\ref{basic-independentpoly}(3).
					\item For $a_1 =7$, we must have $(a_2, a_3)  \in \{(7,2), (14, 8)\}$.
					\item For $a_1 =8$, we must have $(a_2, a_3) \in \{(16,9), (20, 14)\}$.  However, $a_1 =8, a_2 =20, a_3 =14$ is not possible by Lemma~\ref{basic-independentpoly}(5).  	\end{enumerate}  
				
				\item Let  $a_2^2=\frac{4a_3(a_1-2)^2}{a_1-3}.$ Then, following the same argument as in the proof of (2) of this lemma, we find that:
				\begin{enumerate}
					
					\item For $a_1 =5$, $a_2 =6 $ and $a_3 =2$.
					\item 	For $a_1 =6$, $a_2 =8 $ and $a_3 =3$.
					\item 	For $a_1 =7$, $(a_2, a_3) \in \{ (5,1), (10,4), (15,9)\}$.
					\item 	For $a_1 =8$, $a_2=12$ and $ a_3=5$.
				\end{enumerate} 
				
				\item Let $\frac{4a_3(a_1-2)^2}{a_1-3}<a_2^2<4a_1a_3.$ Then arguing as in the proof of (1), we find:
				\begin{enumerate}
					\item There is no reduced independence polynomial corresponding to  $a_1 = 5 $ or $6$.
					\item  For $a_1=7$, we have $25 a_3 < a_2 ^2 < 28 a_3$  and by Lemma~\ref{basic-independentpoly}(4), $3 \leq a_2  \leq 16 $. This gives that    $1 \leq a_3  \leq 10$.  Putting all possible values of $a_3$, it can be seen that   $(a_2, a_3) \in \{(9,3), (16,10)\}$. However, $(a_1,a_2,a_3) =(7,16,10)$ is not possible by Lemma~\ref{basic-independentpoly}(4).
					\item  For $a_1=8$, we have $28.8 a_3 < a_2 ^2 < 32 a_3$ and it follows from Lemma~\ref{basic-independentpoly}(5) that $a_2 < 22$.  Now $ 28.8 a_3 < a_2 ^2$ gives that $a_3 \leq 15$. Putting all possible values of $a_3$, we get that
					
					$(a_2, a_3) \in \{(11,4), (17, 10), (18,11), (19,12), (20,13), (21,14), (21,15)\}.$  It is not possible to have $(a_2, a_3) \in \{(20,13), (21,14), (21,15)\}$, in view of Lemma~\ref{basic-independentpoly}(5). Therefore,    
					$(a_2, a_3) \in \{(11,4), (17, 10), (18,11), (19,12)\}.$
				\end{enumerate}  
			\end{enumerate}
		\end{proof}
		The following lemma determining the Julia set of the reduced independence polynomial in all the fives sub cases forms the basis for the proof of Theorem~\ref{bicritically real-two real fixed points}(1).
		\begin{lemma}
			Let  $P(z)=a_1 z +a_2 z^2 +a_3 z^3$ be the reduced independence polynomial. If $a_1 > 4$ and $ 4a_3 (a_1 -1) < a_2 ^2 < 4 a_1 a_3 $  then the following are true.
			\begin{enumerate}
				\item If $4a_3(a_1-1)<a_2^2<\frac{a_3(2a_1-3)^2}{a_1-2}$, then the Julia set of $P$ is disconnected but not totally disconnected.
				\item If $a_2^2=\frac{a_3(2a_1-3)^2}{a_1-2}$, then for $a_1 \leq 6$, the Julia set of $P$ is connected, and for $a_1 >6$, it is disconnected but not totally disconnected.
				\item If $\frac{a_3(2a_1-3)^2}{a_1-2}<a_2^2<\frac{4a_3(a_1-2)^2}{a_1-3}$, then the Julia set of $P$ is disconnected but not totally disconnected.
				\item If $a_2^2=\frac{4a_3(a_1-2)^2}{a_1-3}$, then for $a_1 \leq 7$, Julia set of $P$ is connected, and for $a_1 >7$, it is disconnected but not totally disconnected.
				\item  If $\frac{4a_3(a_1-2)^2}{a_1-3}<a_2^2<4 a_1 a_3$, then Julia set of $P$ is disconnected except for $(a_1, a_2, a_3)=(7,9,3)$ or $(8, 11,4)$.
			\end{enumerate} 
			\label{fivecases-dynamics}
		\end{lemma} 
		\begin{proof}
			
			\begin{enumerate}
				\item It follows from Lemma~\ref{bicriticalreal-twofixedpoints-twononrealzero}(1) that, there are exactly three possible reduced independence polynomials for $4< a_1 \leq8$, namely $ 6z^3+11z^2+6z,  7z^3+13z^2+7z$ and $ 8z^3+15z^2+8z$. The values of the smaller fixed point $\delta_1$, the larger critical point $c_2$ and $P(c_2)$ are given in the Table~\ref{Table1}, from which, we have that $P(c_2)<\delta_1$  whenever $4 <a_1 \leq 8$. The graph of $6x^3 +11x^2 +6x$ is given in Figure~\ref{bicriticallyreal-tworealfixedpoints-case1} to demonstrate that $P(c_2)< \delta_1$.
				\renewcommand{\arraystretch}{1.4} 
				\begin{longtable}{|m{2.9cm}|m{0.65cm}|m{1.5cm}|m{1.4cm}|m{2.0cm}|} 
					\hline
					\centering $P(z)$ & \centering $\delta_1$ & \centering  $c_2$ & \centering $P(c_2)$ &  {Remark } \\ 
					\hline
					$6z^3+11z^2+6z$& $-1$ & $\frac{-11+\sqrt{13}}{18}$ & $-1.024$ & $P(c_2)<\delta_1$  \\
					\hline
					$7z^3+13z^2+7z$ & $-1$ &  $\frac{-13+\sqrt{22}}{21}$ & $-1.168$ & $P(c_2)<\delta_1$  \\
					\hline
					$8z^3+15z^2+8z$ & $-1$ & $\frac{-15+\sqrt{33}}{24}$ & $-1.313$ & $P(c_2)<\delta_1$  \\
					\hline
					
					\caption{The value of the smaller fixed point $\delta_1,$ the larger critical point $c_2$ and the corresponding critical value, rounded off  to three decimal places of $P(z)=a_1 z+a_2 z^2+a_3 z^3$ when $4a_3(a_1-1)<a_2^2<\frac{a_3(2a_1-3)^2}{a_1-2}$. }
					\label{Table1}
				\end{longtable} 		
				It follows from Lemma~\ref{bicriticalreal-tworealfixedpoint-twononrealzero-lemma}(1) that  the  Julia set of $P$ is disconnected whenever $4 < a_1 \leq 8$. Also Lemma~\ref{bicriticalreal-tworealfixedpoint-twononrealzero-lemma}(2) gives that the Julia set is disconnected for $a_1 \geq 9$. However, Lemma \ref{five-cases}(1) shows that $P$ has an attracting domain corresponding to $\delta_2$ for each $a_1 >4$. Since the Fatou set of $P$ contains this attracting domain, the Julia set is not totally disconnected. 
				
				\item 	Using the condition $a_2^2=\frac{a_3(2a_1-3)^2}{a_1-2}$, we have $c_2 = \frac{-a_1}{3 \sqrt{a_3 (a_1 -2)}}$ and $\delta_1 = \frac{1-a_1}{\sqrt{a_3 (a_1 -2)}}$ and consequently,  $\delta_1-P(c_2)=\frac{(2a_1-3)^2(a_1-6)}{27(a_1-2)\sqrt{a_3(a_1-2)}}$. This implies that  $\delta_1 < P(c_2)$ for  $4 < a_1 < 6$;
				  $\delta_1 = P(c_2)$ for $a_1=6$ and $P(c_2)<\delta_1$ for $a_1>6$. Therefore if $4 < a_1 \leq 6$ then the Julia set of $P$ is connected, otherwise it is disconnected, by Lemma \ref{$c_2 escape$}(1). However, Lemma \ref{five-cases}(2) shows that for $a_1>4$, $P$ has an attracting domain corresponding to the super-attracting fixed point $\delta_2$. Since the Fatou set of $P$ contains  this attracting domain, its Julia set is not totally disconnected whenever $a_1>6$.    		
				
				\item If $\frac{a_3(2a_1-3)^2}{a_1-2}<a_2^2<\frac{4a_3(a_1-2)^2}{a_1-3}$ then there are exactly three possible reduced independence polynomials for $4< a_1 \leq8$, namely $ 2z^3+7z^2+7z,  8z^3+14z^2+7z$ and $ 9z^3+16z^2+8z$ (see Lemma~\ref{bicriticalreal-twofixedpoints-twononrealzero}(3)).
				\renewcommand{\arraystretch}{1.4} 
				\begin{longtable}{|m{2.9cm}|m{0.65cm}|m{1.5cm}|m{1.4cm}|m{2.0cm}|} 
					\hline
					\centering $P(z)$ & \centering  {$\delta_1$} & \centering {$c_2$} & \centering {$P(c_2)$} &  Remark \\ 
					\hline
					$2z^3+7z^2+7z$& $-2$ & $\frac{-7+\sqrt{7}}{6}$ & $-2.158$ & $P(c_2)<\delta_1$  \\
					\hline
					$8z^3+14z^2+7z$ & $-1$ &  $\frac{-7+\sqrt{7}}{12}$ & $-1.079$ & $P(c_2)<\delta_1$  \\
					\hline
					$9z^3+16z^2+8z$ & $-1$ & $\frac{-16+\sqrt{40}}{27}$ & $-1.226$ & $P(c_2)<\delta_1$  \\
					\hline

					\caption{The value of the smaller fixed point $\delta_1,$ the larger critical point $c_2$ and the corresponding critical value, rounded off to three decimal places of $P(z)=a_1 z+a_2 z^2+a_3 z^3$ when $\frac{a_3(2a_1-3)^2}{a_1-2}< a_2^2< \frac{4a_3(a_1 -2)^2}{a_1 -3}$.}
					\label{Table (iii)}
				\end{longtable}
				
				It follows from  Table \ref{Table (iii)}  that the critical value corresponding to the larger critical point $P(c_2)$ is less than the smaller fixed point $\delta_1$. The Julia set of $P$ is disconnected by Lemma~\ref{bicriticalreal-tworealfixedpoint-twononrealzero-lemma}(1). For $a_1 \geq 9$, Lemma~\ref{bicriticalreal-tworealfixedpoint-twononrealzero-lemma}(2) gives that the Julia set of $P$ is disconnected. Note that $a_1 \geq 7$ in this case. Lemma \ref{five-cases}(3) gives that $P$ has an attracting domain corresponding to the attracting fixed point $\delta_2$. Since the Fatou set contains this attracting domain, its Julia set is not totally disconnected.
				\item 	Using the condition $a_2^2=\frac{4a_3(a_1-2)^2}{a_1-3}$, we obtain $\delta_1 =\frac{1-a_1}{\sqrt{a_3 (a_1 -3)}}$ and $c_2 = \frac{4-2 a_1 +\sqrt{a_1 ^2 -7 a_1 +16}}{3 \sqrt{a_3 (a_1 -3)}}$. Thus, $\delta_1-P(c_2)=\frac{(2a_1^3-21a_1^2+24a_1+47)+2(a_1^2-7a_1+16)^\frac{3}{2}}{27(a_1-3)\sqrt{a_3(a_1-3)}}$ (see Equation~\ref{criticalvalue-bicritical-real}). Note that  for $4 < a_1 < 7$, $\delta_1 < P(c_2)$, and if  $a_1=7$ then $\delta_1 = P(c_2)$. The graph of the corresponding polynomial $2x^3+6x^2+5x$ is given in
				Figure~\ref{bicriticallyreal-tworealfixedpoints-case3} to demonstrate that $P(c_2)>\delta_1$. The numerator of the right-hand side expression  is an increasing function of $a_1$ in $(7, \infty)$ and its value is $0$ for $a_1 =7$. Thus,   $P(c_2)<\delta_1$ whenever $a_1>7$. Therefore if $4 < a_1 \leq 7$ then the Julia set of $P$ is connected; otherwise it is disconnected, by Lemma \ref{$c_2 escape$}(1). However, Lemma \ref{five-cases}(4) shows that for $a_1>4$, $P$ has a parabolic domain corresponding to the parabolic fixed point $\delta_2$. Since the  Fatou set of $P$ contains this parabolic domain, its Julia set is not totally disconnected whenever $a_1 >7$.
				
				\item It follows from (see Lemma~\ref{bicriticalreal-twofixedpoints-twononrealzero}(5)) that there are exactly five possible reduced independence polynomials for $4< a_1 \leq8$. These are $ 3z^3+9z^2+7z,  4z^3+11z^2+8z, 10z^3+17z^2+8z,  11z^3+18z^2+8z$ and $  12z^3+19z^2+8z$. From the Table \ref{Table (v)}, we have  $P (c_2)\geq\delta_1$ for $(a_1, a_2, a_3)=(7,9,3)$ or $(8, 11,4)$ and the Julia set of $P$ is connected in these cases by Lemma \ref{$c_2 escape$}(1). The Julia set is disconnected otherwise.  Lemma \ref{$c_2 escape$}(2) also shows that the Julia set of $P$ is disconnected for $a_1\geq 9$.
				\renewcommand{\arraystretch}{1.4}    		
				\begin{longtable}{|m{2.9cm}|m{0.65cm}|m{1.5cm}|m{1.4cm}|m{2.0cm}|} 
					\hline
					\centering{\textbf{$P(z)$}} &\centering \textbf{$\delta_1$} &\centering{ \textbf{$c_2$}} & \centering{\textbf{$P(c_2)$}} &  Remark \\ 
					\hline
					$3z^3+9z^2+7z$& $-2$ & $\frac{-3+\sqrt{2}}{3}$ & $-1.629$ & $P(c_2)>\delta_1$  \\
					\hline
					$4z^3+11z^2+8z$ & $-\frac{7}{4}$ & $-\frac{1}{2}$ & $-1.75$ & $P(c_2)=\delta_1$  \\
					\hline
					$10z^3+17z^2+8z$ & $-1$ & $-\frac{1}{3}$ & $-1.148$ & $P(c_2)<\delta_1$  \\
					\hline
					$11z^3+18z^2+8z$ & $-1$ & $\frac{-18+\sqrt{60}}{33}$ & $-1.078$ & $P(c_2)<\delta_1$  \\
					\hline
					$12z^3+19z^2+8z$ & $-1$ & $\frac{-19+\sqrt{73}}{36}$ & $-1.015$ & $P(c_2)<\delta_1$  \\
					\hline
					\caption{The value of the smaller fixed point $\delta_1,$ the larger critical point $c_2$ and the corresponding critical value, rounded off to three decimal places of $P(z)=a_1 z+a_2 z^2+a_3 z^3$ when $ \frac{4a_3(a_1 -2)^2}{a_1 -3} < a_2^2< 4a_1 a_3$.}
					\label{Table (v)}
				\end{longtable}
			\end{enumerate}
		\end{proof}
		
		Now the proof of  Theorem~\ref{bicritically real-two real fixed points}(1) is presented.

		\begin{proof}[Proof of Theorem~\ref{bicritically real-two real fixed points}(1)]
			If $a_1=4$ then $12a_3<a_2^2<16a_3$. By Lemma~2.8(1),
			$(a_2,a_3)\in\{(3,1),(4,1),(5,2)\}$, and only $(5,2)$ satisfies both inequalities.
			Hence $P(z)=4z+5z^2+2z^3$, so $a_2=2a_1-3$ and $a_3=a_1-2$. Since $P(-1)=-1$
			and $P'(-1)=0$, the point $-1$ is a super-attracting fixed point of $P$ (see
			Figure~\ref{bicriticallyreal-tworealfixedpoints-case2}). By
			Theorem~\ref{IA-JS}, $\mathcal{A}(G)$ is the disjoint union of $\mathcal{J}(P)$
			and $\bigcup_{k\ge1}\mathrm{Roots}(I_{G^k})$. The latter set is nothing but the backward orbit of $-1$ under $P$ and $-1$ is a super-attracting fixed point of $P$ .  
			Therefore $\mathcal{A}(G)$ is disconnected. The Fatou set of $P$ has at least two
			components, namely the basin of $\infty$ and the attracting domain of $-1$, so
			$\mathcal{J}(P)$ is not totally disconnected. As $\mathcal{J}(P)\subset\mathcal{A}(G)$,
			the attractor $\mathcal{A}(G)$ is not totally disconnected.
			\par
			Let $a_1>4$. By Theorem~\ref{IA-JS}, the statements (a), (c), (d) and (e)
			follow from Lemma~\ref{fivecases-dynamics}(1), (3), (4) and (5), respectively.
			
			To prove (b), let $a_2^2=\frac{a_3(2a_1-3)^2}{a_1-2}$. For $a_1\le 6$ the possible
			values of $(a_1,a_2,a_3)$ are $(5,7,3)$ and $(6,9,4)$ (see
			Lemma~\ref{bicriticalreal-twofixedpoints-twononrealzero}(2)), and $\mathcal{J}(P)$
			is connected by Lemma~\ref{fivecases-dynamics}(2). By Theorem~\ref{IA-JS},
			$\mathcal{A}(G)$ is the disjoint union of $\mathcal{J}(P)$ and
			$\bigcup_{k\ge1}\mathrm{Roots}(I_{G^k})$ in these cases, so $\mathcal{A}(G)$ is
			disconnected by similar reason as in the above case $a_1=4$. For $a_1>6$, $\mathcal{J}(P)$ is disconnected by
			Lemma~\ref{fivecases-dynamics}(2), and hence $\mathcal{A}(G)$ is disconnected by
			Theorem~\ref{IA-JS}. Finally, by Lemma~\ref{five-cases}, the larger fixed point
			$\delta_2$ is super-attracting for $a_1>4$. The Fatou set of $P$ therefore has at
			least two components, the basins of $\infty$ and of $\delta_2$, so $\mathcal{J}(P)$
			is not totally disconnected.
			This completes the proof of (b).
		\end{proof}

		\begin{Remark} Note that
			$(a_1, a_2,a_3)=(9,13,5)$ satisfies $\frac{4a_3(a_1-2)^2}{a_1-3}<a_2^2<4 a_1 a_3$, and by Lemma~\ref{bicriticalreal-tworealfixedpoint-twononrealzero-lemma}(2), the   larger critical point escapes. It can be seen that the smaller critical point is $c_1 =-1.255$ approximately and $P^2 (c_1) $ is less than the smaller fixed point. By Observation~\ref{obs 3.1}(3), $c_1$ also escapes. It now follows from Lemma~\ref{connected-Juliaset}(2) that the Julia set of $9z+ 13z^2 +5 z^3$ is totally disconnected. The independence attractor is totally disconnected as it is the same as this Julia set in this case. 
			\label{totallydisconnected-example-9-13-5} 
		\end{Remark}
		\noindent \underline{\textbf{Case II - $a_2^2=4a_1a_3$ - Only one  non-zero zero of $P$ }}:\\
		We need a lemma to prove Theorem~\ref{bicritically real-two real fixed points}(2).
		\begin{lemma}
			If $a_2^2=4a_1a_3$ then the Julia set of $P(z)=a_1z+a_2z^2+a_3z^3$ is connected for $a_1\leq9$ and totally disconnected for $a_1>9$. Moreover, the Julia set of $P$ is a line segment for $a_1=9$. 
			\label{bicriticalreal-tworealfixedpoint-onezero-lemma}  
		\end{lemma}
		\begin{proof}
			The non-zero zero of $P$ is $-\frac{a_2}{2a_3}$, and this is a critical point, which must be $c_1$.  Therefore $c_1$ lies in the Julia set of $P$ as $0 \in \mathcal{J}(P)$ and $\mathcal{J}(P)$ is backward invariant.  
			
			Note that $c_1 <\delta_2$, otherwise, i.e.,  if $c_1 \geq \delta_2$ then by Observation \ref{obs 3.1}(2) $c_1 \geq P(c_1)$, i.e., $c_1 \geq 0$, which is not true.  Therefore $\delta_1 <c_1 <\delta_2<c_2$ (see Equation~\ref{critical-fixed-points}). Using $a_2 ^2=4a_1 a_3$, we get $c_2=-\frac{\sqrt{a_1}}{3 \sqrt{a_3}}$ and $\delta_1=-\frac{(1+\sqrt{a_1})}{\sqrt{a_3}}$, giving that $P(c_2)-\delta_1=\frac{27-(4a_1-27)\sqrt{a_1}}{27\sqrt{a_3}}$. 
			Therefore,  for $a_1<9$, $P(c_2)> \delta_1$; for $a_1=9$, $P(c_2)=\delta_1$, and  $P(c_2)<\delta_1$ 
			for $a_1>9$.  
			\par Let $a_1 <9$. As $a_1 \geq 4$ in this case,  we have $P(c_2)> \delta_1$ for $4\leq a_1<9$.  It follows from Observation \ref{obs 3.1}(1) that  $P([\delta_1,0]) =P([\delta_1, c_1]) \cup P([c_1, 0])=[\delta_1, 0] \cup [P(c_2), 0]$, which is nothing but $[\delta_1, 0]$. Thus $P^{n}([\delta_1,0])=[\delta_1,0]$ for all $n \in \mathbb{N}$. Figure~\ref{a_1<9} demonstrates this situation for $a_1 =7$. This implies that $c_2$ is in the filled-in Julia set of $P$. By Lemma~\ref{connected-Juliaset}(1), the Julia set of $P$ is connected.
			\begin{center}
				\begin{figure}[h!]
					\begin{center}
						\begin{subfigure}{.52\textwidth}
							\includegraphics[width=1.05\linewidth]{ 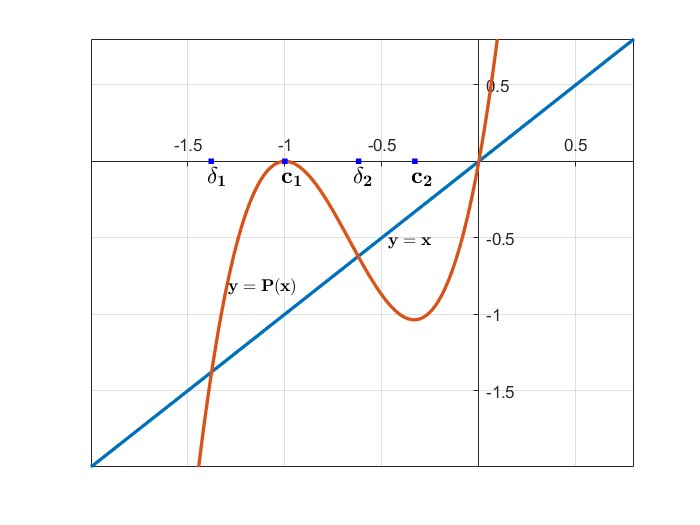}
							\caption{Graph of $P(x)=7x^3+14x^2+7x$}
							\label{a_1<9}
						\end{subfigure}%
						\begin{subfigure}{.52\textwidth}
							\centering
							\includegraphics[width=1.05\linewidth]{ 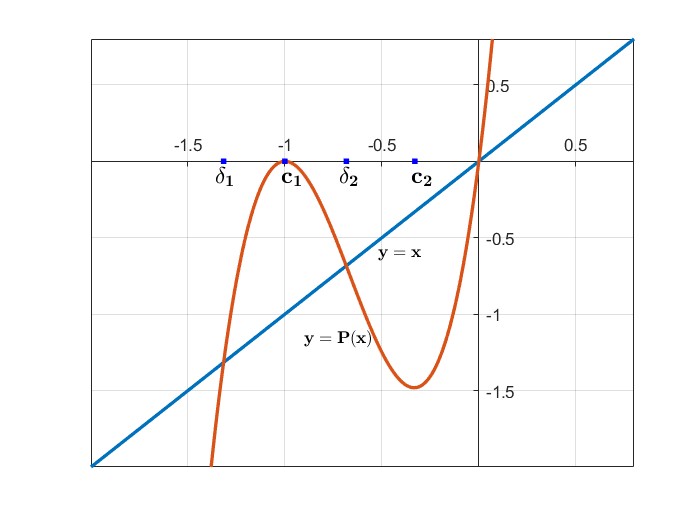}
							\caption{Graph of $P(x)=10x^3+20x^2+10x$}
							\label{a_1>9}
						\end{subfigure}
					\end{center}
					\caption{For $a_2 ^2=4a_1 a_3$ (a) $P(c_2)> \delta_1$ and (b)  $P(c_2)< \delta_1$.}	
				\end{figure}	
				
			\end{center}
			\par For $a_1=9$, $P(c_2)=\delta_1$, and it follows  from Observation \ref{obs 3.1}(1) that $P([\delta_1, 0])=[\delta_1,0]$. Every point in $[\delta_1, 0] \setminus \{P(c_1), P(c_2)\}$ has exactly three pre-images in $[\delta_1, 0]$. There are also three pre-images of $P(c_1)$ and $ P(c_2)$ counting multiplicity in $[\delta_1, 0]$. Since the degree of $P$ is $3$,  $P^{-1}([\delta_1,0])=[\delta_1,0]$. Therefore $[\delta_1,0]$ is completely invariant under $P$, and by Theorem 4.2.2 of ~\cite{Beardon_book}, the Julia set of $P$ is contained in $[\delta_1,0]$. Since $P(c_1)=0$, $P(c_2)=\delta_1$ and $0,\delta_1$ are repelling fixed points of $P$,  the Fatou set of $P$ has only one Fatou component, namely the basin of $\infty$. As $P^{n}([\delta_1,0])=[\delta_1,0]$ for all $ n $, no point of $[\delta_1,0]$ is in the basin of $\infty$. Therefore $[\delta_1,0]$ is contained in the Julia set of $P$. Hence the Julia set of $P$ is $[\delta_1,0]$.
			\par
			For $a_1>9$, $P(c_2) <\delta_1$, and it follows that $c_2$ is in the basin of $\infty$ (see Observation \ref{obs 3.1}(3)). By Observation \ref{obs 3.1}(1), every inverse image  of each point of   $[\delta_1,0]$ is in $[\delta_1,0]$ (see Figure~\ref{a_1>9} for $a_1 =10, a_3=10$). Therefore $P^{-1}([\delta_1,0])=[\delta_1,0]$. The point at $0$ is a repelling fixed point and is in the Julia set of $P$.  By Theorem 4.2.7, \cite{Beardon_book}, $\mathcal{J}(P)=\overline{\cup_{n=0}^{\infty}P^{-(n)}(0)}\subseteq [\delta_1,0]$. We assert that $\mathcal{J}(P)$ is totally disconnected. If not, then there is an open interval $(w-\epsilon,w+\epsilon) \subset \mathcal{J}(P)$ for $w \in \mathcal{J}(P)$ and some positive $\epsilon$. Since $P^{(-n)}(c_2)$ is in the Fatou set of $P$, we have  $P^{(-n)}(c_2) \notin (w-\epsilon,w+\epsilon)$ for any  $n\geq0 $. Using the backward invariance of $[\delta_1, 0]$ under $P$ and the fact that $c_2 \in [\delta_1, 0]$, we get  $w \notin \overline{\cup_{n=0}^{\infty}P^{-(n)}(c_2)}$. Here $c_2$ is a non-exceptional point whose backward orbit does not accumulate at $w$, a point in the Julia set. This is a contradiction to Theorem 4.2.7, \cite{Beardon_book}.
			Therefore $\mathcal{J}(P)$ is totally disconnected.
		\end{proof}   
		\begin{Remark}
			It may be interesting to determine the Fatou set of $a_1 z+a_2 z^2 +a_3 z^3$ when $a_2 ^2 =4a_1 a_3$ and $a_1 \in \{4,5,6,7,8\}$.
		\end{Remark}
		\begin{proof}[Proof of Theorem~\ref{bicritically real-two real fixed points}(2)]
			The proof follows from Theorem~\ref{IA-JS} and Lemma~\ref{bicriticalreal-tworealfixedpoint-onezero-lemma}.
		\end{proof}
		\noindent \underline{\textbf{Case III - $a_2^2 >4a_1a_3$ - Two distinct real zeros of P}}:\\
		\begin{proof}[Proof of Theorem~\ref{bicritically real-two real fixed points}(3)]
			Since  $a_2^2>4a_1a_3$, there are two distinct real zeros of $P$ and those  are  $\zeta_1=\frac{-a_2 - \sqrt{a_2 ^2 -4 a_1 a_3}}{2a_3}$ and $\zeta_2=\frac{-a_2 + \sqrt{a_2 ^2 -4 a_1 a_3}}{2a_3}$. By Rolle's theorem, there is a critical point of $P$ in $(\zeta_1, \zeta_2)$, which must be $c_1$ by Observation~\ref{obs 3.1}(2) (see Figure~\ref{CaseIII-fig} for $(a_1,a_2,a_3)=(5,5,1)$). Since $P$ is increasing in $(\zeta_1, c_1)$, we have $P(c_1)>0$. Thus, $c_1$ is in the basin of $\infty$ by Lemma~\ref{Juliaset-prop}(2). Therefore, by Lemma~\ref{connected-Juliaset}(1) the Julia set of $P$ is disconnected. We are done by Theorem~\ref{IA-JS}.
		\end{proof}
		\begin{center}
			\begin{figure}[h!]
				\centering
				\includegraphics[width=0.6\linewidth]{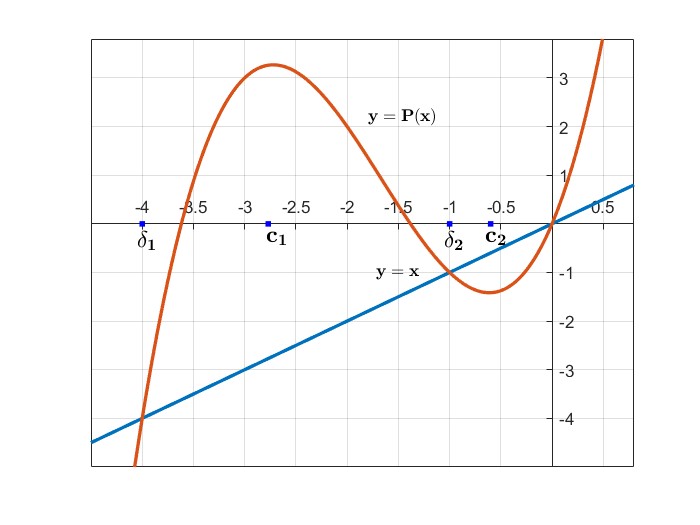}
				\caption{Graph of $P(x)=x^3+5x^2+5x$, represents the situation $a_2 ^2 > 4 a_1 a_3$.}
				\label{CaseIII-fig}
			\end{figure}
		\end{center}
		\begin{Remark}
			In the proof of Theorem~\ref{bicritically real-two real fixed points}(3), the Julia set is totally disconnected  whenever $P(c_2) < \delta_1$. It is yet to be known whether it is always so.
		\end{Remark}
		\begin{Remark}
			Note that $(a_1, a_2, a_3)=(7,6,1)$ satisfies $a_2 ^2 > 4 a_1 a_3$. Here the critical value corresponding to the smaller critical point is positive.  The larger critical point $c_2 =-0.7$ approximately and $P^2(c_2)>0$. Therefore both the critical points escape  as all positive real numbers escape (by Lemma~\ref{Juliaset-prop}(2)). It now follows from Lemma~\ref{connected-Juliaset} (2) that the Julia set of $7z+6z^2+z^3$ is totally disconnected. Since in this case the independence attractor is the same as this Julia set, the independence attractor is totally disconnected.
			\label{totallydisconnected-example-7-6-1}
		\end{Remark}
		We now present the proof of Theorem~\ref{examples}.
		\begin{proof}[Proof of Theorem~\ref{examples}]
			We are going to provide graphs with a single $3$-independent set.  For every graph $G$, all the repelling and parabolic fixed points of $P_G$ are in $\mathcal{A}(G)$. In particular,  $0 \in \mathcal{A}(G)$.
			
			The proof of each case of this theorem involves constructing  a sequence  $G(n)$ of graphs for which $P_{G(n)}$ has a non-zero repelling or parabolic fixed point  $z_n$ such that $\lim\limits_{n \to \infty}z_n =\infty$. This gives that the diameter of $\mathcal{A}(G(n)) \geq |z_n|$ goes to $\infty$ as $n \to \infty$. The graphs are obtained by removing required number of edges (but not the vertices on which these are incident) from a complete graph on suitably many vertices. 
			\par 
			Let $n>3$ and $G_1 (n)$ be the graph  obtained from the complete graph on $n$ vertices by removing only three edges of a triangle. Then $P_{G_1(n)}(z)=nz+3z^2 +z^3$, for which $z_n =-\frac{3}{2}+i \sqrt{n-\frac{13}{4}} $ is a repelling fixed point.  It is easy to verify that each $G_1 (n)$ is bicritically non-real.
			
			\par To construct $G_2 (n)$, consider the complete graph on $3 n^2$ vertices with $n>1$, say   $v_j, j =1,2,3,\cdots, 3n^2$. Denote an edge joining $v_i$  and $ v_j$ by $v_iv_j$. First remove  the three edges $v_1 v_2, v_2 v_3, v_1 v_3$.   Then remove $3n -3$ edges  $v_4 v_j$ for $j=5,6,7,\cdots 3n+1$. Then the reduced independence polynomial of the  resulting graph $G_2 (n)$ is $ 3n^2 z+ 3n z^2+z^3$. Each $G_2 (n)$ is clearly unicritical. The point $z_n = \frac{-3n +i \sqrt{3n^2 -4}}{2}$ is a repelling fixed point of $P_{G_2 (n)}$ (see Remark~\ref{unicritical-fixedpoints}). 
			\par 
			In order to construct $G_3 (n), n>1$,	consider the complete graph on $ n^2 +1$ vertices, say   $v_j, j =1,2,3,\cdots,  n^2 +1$. First remove  the three edges $v_1 v_2, v_2 v_3, v_1 v_3$. Then remove the edges  $v_4v_j$ for $j=5,6,7,\cdots, 2n+1$. Thus, $2n$ edges are removed. The reduced independence polynomial of the resulting graph is $(n^2 +1) z+ 2nz^2 +z^3$ and that has a parabolic fixed point at $ -n$ . It is not difficult to check that each $G_3 (n)$ is bicritically real.
		\end{proof}
		\textbf{Acknowledgement:} Moumita Manna is supported by  University Grants Commission, Govt. of India through a Senior Research Fellowship.
		

\begin{thebibliography}{00}
			
			\bibitem{line} S. Barik, T. Nayak and A. Pradhan, Graphs whose independence fractals are line segments, \emph{ Bull. Malays. Math. Sci. Soc.}  44 (2021), no.~1, 55-78.
			
			\bibitem{Beardon_book}  A.F. Beardon, Iteration of Rational Functions, Grad. Texts in Math. 132, Springer-Verlag, 1991. 
			
			\bibitem{ivona-2020}I. Bez\'akov\'a, A. Galanis, L. Goldberg and D. \v Stefankovi\v c, Inapproximability of the independent set polynomial in the complex plane, \emph{SIAM J. Comput.} 49 (2020), no.~5, STOC18-395--STOC18-448.
			
			\bibitem{Bollobas} B.~Bollob\'as, Extremal Graph Theory, Academic Press,
			London--New York, 1978.
			
			
			\bibitem{brownetal2003} J.~I.~Brown, C.~A.~Hickman and R.~J.~Nowakowski, The independence fractal of a graph, \emph{J. Comb. Theory Ser. B} 87 (2003), 209--230.	
			
			\bibitem{brownetal2004} J.~I.~Brown, C.~A.~Hickman and R.~J.~Nowakowski, On the location of the roots of independence polynomials, \emph{J. Algebraic Combin.} 19 (2004), no.~3, 273-282. 
			%
			\bibitem{Buff-2002} X. Buff, On the Bieberbach conjecture and holomorphic dynamics, \emph{Proc. Amer. Math. Soc.} 131 (2003), no.~3, 755-759. 
			
			
			
			
			
			
				\bibitem{hickman2001} C.~A.~Hickman, Roots of chromatic and independence polynomials, Ph.D. Thesis, Dalhousie University, 2001.
			\bibitem{moumita-2025}  G. Khetawat, M. Manna, and T. Nayak, Circles and line segments as independence attractors, https://doi.org/10.48550/arXiv.2505.20898, 2025
			
			
			\bibitem{rotational-symmetry-2024} T. Nayak and S. Pal, Julia Sets of Rational Maps with Rotational Symmetries. \emph{Complex Anal. Oper. Theory} 19 (2025), no.~1, 11.	
			\bibitem{hans-2019}
			H. Peters and G. Regts, On a conjecture of Sokal concerning roots of the independence polynomial, \emph{ Michigan Math. J.} 68 (2019), no.~1, 33-55.
			
			\bibitem{Pom 1975} C. Pommerenke, Univalent Functions, Vandenhoeck and Ruprecht, Gottingen, 1975. 
			
			\bibitem{scot-sokal2005}
			A. D. Scott and A. D. Sokal, The repulsive lattice gas, the independent-set
			polynomial, and the Lovász local lemma, \emph{ J. Stat. Phys.} 118 (2005), no.~5-6, 1151-1261.
			\bibitem{leeyang1952} C. N. Yang and T. D. Lee, Statistical theory of equations of state and phase transitions.  I. Theory of condensation, \emph{Phys. Rev.} (2) 87 (1952), 404-409. 
			
		\end{thebibliography}
	\end{document}